\documentclass[12pt]{amsart}%
\usepackage{amsfonts}
\usepackage{epsfig}
\usepackage{graphicx}
\usepackage{amsmath}
\usepackage{amsfonts, color, soul}
\usepackage{amssymb}
\usepackage{latexsym}
\usepackage{rotating}
\usepackage{pstricks, pst-node, pst-text, pst-3d}
\usepackage[arrow, matrix, curve, xdvi, dvips]{xy}
\usepackage[text={5.35in,8.3in}]{geometry}
\usepackage{stmaryrd}
\usepackage{amsbsy}
\usepackage{bm}
\usepackage{dsfont}
\usepackage{pdfsync}
\usepackage{color}
\usepackage[yyyymmdd]{datetime}
\usepackage{ulem}

\input{amssym.def}
\newtheorem{theorem}{Theorem}
\newtheorem{proposition}{Proposition}
\newtheorem{corollary}{Corollary}
\newtheorem{lemma}{Lemma}

\newtheorem{definition}{Definition}
\theoremstyle{remark}

\newcommand{\Ad}{\text{\rm Ad}}
\newcommand{\St}{\text{\rm St}}
\newcommand{\Gr}{\text{\rm Gr}}
\newcommand{\Id}{\text{\rm Id}}

\newcommand{\ad}{\text{\rm ad}}

\newcommand{\SO}{\text{\rm SO}}
\newcommand{\SU}{\text{\rm SU}}
\newcommand{\Sp}{\text{\rm Sp}}
\newcommand{\Tr}{\text{\rm Tr}}
\newcommand{\Kl}{\text{\rm Kl}}
\newcommand{\Orth}{\text{\rm O}}
\newcommand{\U}{\text{\rm U}}

\newcommand{\fg}{\mathfrak {g}}

\newcommand{\fk}{\mathfrak {k}}
\newcommand{\fp}{\mathfrak {p}}
\newcommand{\fso}{\mathfrak{so}}

\newcommand{\fu}{\mathfrak{u}}
\newcommand{\fsp}{\mathfrak{sp}}
\newcommand{\fgl}{\mathfrak{gl}}

\newcommand\cV{\mathcal{V}}

\newcommand\cH{\mathcal{H}}

\newcommand\ep{{\epsilon}}

\newcommand\ii{{\mathbf i}}
\newcommand\jj{{\mathbf j}}
\newcommand\kk{{\mathbf k}}

\newcommand\IC{{\mathbb C}}

\newcommand\IH{{\mathbb H}}
\newcommand\IR{{\mathbb R}}

\definecolor{amethyst}{rgb}{0.6, 0.4, 0.8}

\begin{document}

\title{Extremal curves on Stiefel and Grassmann manifolds}


\author{V. Jurdjevic}
\address{ Department of Mathematics, University of Toronto, 
\newline Toronto, Ontario  M5S 3G3, Canada}
\email{jurdj@math.utoronto.ca}

\author{I. Markina}
\address{Department of Mathematics, University of Bergen, \newline P.O.~Box 7803,
Bergen N-5020, Norway}
\email{irina.markina@uib.no}
\thanks{This work was partially supported by ISP project 239033/F20 of Norwegian Research Council and the work of the second author was also partially supported by joint BFS-TSF mathematics program.}

\author{F. Silva Leite}
\address{Department of Mathematics, University of Coimbra, \newline Apartado 3008, 3001-501 Coimbra, Portugal and Institute of Systems and Robotics, Uni\-versity of Coimbra - P\'olo II, 3030-290 Coimbra, Portugal}
\email{fleite@mat.uc.pt}


\subjclass[2010]{Primary 53C17, 53C22, 53B21, 53C25, 30C80; Secondary 49J15, 58E40.} 
\keywords{Sub-Riemannian geometry, Riemannian and sub-Riemannian geodesics, quasi-geodesic curves, horizontal distributions, Grassmann manifolds, Stiefel manifolds, Lie groups actions on manifolds, Pontryagin Maximum Principle, extremal curves.}
\date{}


\begin{abstract}
This paper  uncovers a large class of left-invariant sub-Rie\-mannian  systems on Lie groups that admit explicit solutions  with certain  properties, and provides geometric origins for  a class of important curves  on Stiefel manifolds, called quasi-geodesics, that project on Grassmann manifolds as Riemannian geodesics. We show that quasi-geodesics are the projections of sub-Riemannian geodesics   generated  by certain  left-invariant distributions  on  Lie groups that act transitively on each Stiefel manifold $\St_k^n(V)$. This result is valid not only for the real Stiefel manifolds in   $V=\IR^n$, but also  for the Stiefels in the Hermitian space $V=\IC^n$ and  the quaternion space $V=\IH^n$.
\end{abstract}

\maketitle


\bibliographystyle{amsplain}

\section{Introduction}

This paper provides geometric origins for  a class of curves  on Stiefel manifolds, called quasi-geodesic,
  that have proved to be particularly important  in solving interpolation problems arising in real applications~\cite{KL}.  We show that quasi-geodesic curves are the projections of sub-Riemannian geodesics   generated  by certain  left-invariant distributions  on  Lie groups $G$ that act on Stiefel manifolds.  
  This quest for  the geometric characterization of quasi-geodesic curves  uncovered a large class of left-invariant sub-Riemannian  systems on Lie groups that admit explicit solutions, in the form that will be made clear  below.   
As a result, the paper is as much about sub-Riemannian structures on Lie groups as it is about quasi-geodesic curves.
  
  The first part of the paper  deals  with sub-Riemannian structures  associated  with homogeneous spaces $M= G/K$  induced by a transitive left action of a  semi-simple Lie group $G$ on a smooth manifold $M$, where $K$ denotes the isotropy subgroup relative to a fixed point~$m_0 \in M$.    
  The sub-Riemannian structures will be defined by a left-invariant distribution $\mathcal H$ generated by a vector space $\fp\subset\fg$ that is  transversal to the Lie algebra $\fk$ of the isotropy group $K$, and   satisfies additional Lie algebraic relations:
  \begin{equation*}
  \fg=\fp\oplus\fk,\quad [\fp,\fk]\subseteq\fp,\quad  \fk\subseteq{[\fp,\fp}].
 \end{equation*} 

The sub-Riemannian metrics  are defined by   bilinear,  symmetric, non-degenerate, $\Ad_G$ invariant forms $\langle . \, , .\rangle$  that are positive definite on $\fp$.  Under additional assumption that $\fp^\perp$, the orthogonal complement relative to the Killing form, is a Lie subalgebra of  $\fg$, we show that the sub-Riemannian geodesics are of the form 
\begin{equation*}
 g(t)=g_0\, e^{t(P+Q)}e^{-tQ},
 \end{equation*}
where $P\in \fp$ and $Q\in\fp^\perp$, see Theorem~\ref{prop:4.3}.
 
In these situations, a Riemannian metric on $M$ is induced by the push forward of the sub-Riemannian metric on $G$. 
 We show that the corresponding Riemannian geodesics on $M$ are the projections of curves in (1)  where $P$ and $Q$ satisfy the additional relation that $P+Q$ is in $\fk^\perp $.  The above findings coincide with the  analogous results on semi-simple Lie groups with an involutive automorphism, as in the theory of symmetric Riemannian spaces, where $\fp^\perp=\fk$, see, for instance,~\cite{Jc2}.

The second part of the paper deals with the application of these results to  the actions of Lie groups on  Stiefel and Grassmann manifolds and their relevance to the quasi-geodesic curves. Rather than dealing exclusively with the real case, we also include   Stiefel and Grassmann manifolds over complex and quaternion algebras.

We realize each Stiefel manifold in two ways as homogeneous manifolds  $G/K$. On  groups $G$ that define  the homogeneous structure we consider three distinct sub-Riemann\-ian structures whose sub-Riemann\-ian geodesics  are described by Theorem 1.  We then single out the sub-classes of sub-Riemannian geodesics that project either to  Riemannian geodesics or to quasi-geodesic curves on the Stiefel manifolds. Additionally, we show that quasi-geodesics are curves of constant curvature relative to the induced Riemannian metric.
 In the process we discovered an interesting fact that  two distinct sub-Riemannian structures  on $G$ can induce isometric Riemannian structures on the homogeneous manifold $G/K$, see Theorem~\ref{th:isom}. 

 We end the paper with a brief discussion of Lagrangian manifolds. First, we show the relation of Theorem 1 to the canonical (symmetric) Riemannian metric,  and secondly, we show that the projections of curves on Stiefel manifolds, descending from the sub-Riemannian curves on $G$, have constant geodesic curvature on the Grassmann manifold relative to its canonical metric. In particular, we show that  the  quasi-geodesic curves on the Stiefel manifolds project onto the Riemannian geodesics on the Grassmannian manifold. 

This article  constitutes a full version of the results announced  in a short paper  \cite{JKL} on Stiefel manifolds  embedded in the Euclidean space $\IR^n$.

\section{Group actions  and sub-Riemannian problems}

\subsection{ Notations and background material} 
We begin  with the notations and basic concepts  that will be necessary throughout  the paper. Details about this background material may be found, for instance, in~\cite{Hl,KMS93,St,Warner}.



In particular, $M$ will denote a smooth (or real analytic) manifold, $TM$ and $T^*M$ denote the tangent and the cotangent bundle of $M$ respectively, and $T_mM$ and $T^*_mM$ denote the tangent and cotangent spaces at $m\in M$. 
If $F$ is a smooth map between smooth manifolds, then  $F_*$ will denote the  tangent map and $F^*$  the dual map.

The set of smooth vector fields on $M$ is denoted by $\Gamma(TM)$. A smooth curve $m(t)$ defined on an open interval $(-\ep,\ep)$ in $M$ is said to be an integral curve, or a solution curve, of $X$  if $\frac{dm}{dt}=X(m(t))$  for all $t\in (-\ep,\ep)$.   A vector field  $X$ is said to be complete if each integral curve of $X$  can be extended to the interval $(-\infty,\infty)$.
 We denote by $\phi_t^X(m)$ the flow on $M$ generated by a complete vector field $X\in\Gamma(TM)$.  Sometimes we will regard a flow as a one parameter group of diffeomorphisms $\{\phi^X_t:\ t\in \mathbb R\}$.  We will make use of the fact that any one parameter group of diffeomorphisms $\phi_t$ on $M$ is generated by the flow of a complete vector field $X$, in the sense that $\phi_t=\phi^X_t$.  In this context,  $X$ is called the infinitesimal generator  of  the group of diffeomorphisms.

Throughout the paper $G$ will denote a Lie group, and $\fg$ will denote its Lie algebra. We think of $\fg$ as the tangent space $T_eG$ at the identity $e\in G$, with the Lie bracket induced by the left-invariant vector fields in $G$:  $[A,B]=[X,Y](e)=(YX-XY)(e)$, where $X(g)=(L_g)_*(A)$ and $Y(g)=(L_g)_*(B)$, $A,B\in\fg$, and $L_g$ is the left translation by $g$ in $G$. We will adopt a short-hand notation and write $gA$ ($Ag$) for the left (right) invariant vector fields.

 If $X(g)=gA$ is a left ($Ag$  right) invariant vector field defined by $A\in\fg$, then its flow $\phi_t^{X}(g)$ is  the left  (right) translate $L_g\phi_t^X(e)$ ($R_g\phi_t^X(e)$) of the flow through the group identity. We will write  $\exp(tA)$, as well as  $e^{tA}$ when convenient, for the curve $\phi_t^X(e)$ with $X(g)=gA$. 




A Lie group $G$ is said to act on a manifold $M$ through the left action $\phi\colon G\times M\rightarrow M$ if $\phi$ satisfies
\begin{equation*}\label{action}
\phi(g,\phi(h,m))=\phi(L_g(h),m)=\phi(gh,m)\quad\text{and}\quad\phi(e,m)=m,
\end{equation*}
for all $g,h$ in $G$, and all $m$ in $M$. We use $\phi_g$ to denote the diffeomorphism $m\mapsto \phi_g(m)$ on $M$.  If $m_0$ is a point in $M$, then  $K_{m_0}=\{g\in G:\ \phi(g,m_0)=m_0\}$ denotes the isotropy group of $m_0$. An isotropy group is a closed subgroup of $G$ and any two isotropy groups are conjugate. If a group $G$ acts transitively on $M$, then $M$ can be regarded as the quotient $G/K_{m_0}$ with $m\in M$ identified with $gK_{m_0}$ whenever $\phi_g(m_0)=m$. Then the natural projection $\pi\colon G\rightarrow G/K_{m_0}$ is given by the map $g\rightarrow \phi_g(m_0)$. 

Each element $A\in\fg$ induces a one-parameter group of diffeomorphisms $\{\phi_{e^{tA}}:\ t\in \IR\}$ on $M$. We will use $X_A$ to denote its infinitesimal generator, so that  $\phi_{e^{tA}}(m)=\phi_{t}^{X_A}(m)$.  We will refer to  $\mathcal F=\{X_A\colon\ A\in\fg\}$ as the family of vector fields on $M$ subordinated to the group  action.  
  Since $G$  acts by the left action on $M$,    $\mathcal F$ is a homomorphic image of the family of right-invariant vector fields on $G$. Therefore, $\mathcal F$ is a finite  dimensional Lie algebra of complete vector fields on $M$. 
  
  If $G$ is connected, then any  $g\in G$ can be written as $g=e^{t_kA_k}\cdots e^{t_1A_1}$ for some elements $A_1,\dots,A_k $ in $\fg$. Then,
  \begin{equation}\label{orbit}
  \phi_g(m)=\phi_{e^{t_kA_k}\cdots e^{t_1A_1}}(m)=\phi^{X_{A_k}}_{t_k}\cdots\phi_{t_1}^{X_{A_1}}(m).
  \end{equation}
  This implies that $M$ is equal to the orbit of $\mathcal F$ through any point $m\in M$  whenever  $G$ is connected, and the action of $G$ on $M$   is transitive. In such a case  any absolutely continuous  curve $m(t)$ in $M$  is a solution of 
\begin{equation*}\label{eq:2}
\frac{dm}{dt}=\sum_{i=1}^k u_i(t)X_{A_i}(m)
\end{equation*}
for some choice of elements $A_1,\dots,A_k$ in $\fg$ and some choice of measurable and bounded functions $u_1(t),\ldots,u_k(t)$.


\subsection{Sub-Riemannian problems on Lie groups}\label{sRproblem}


In what follows we will always suppose that $G$ is a connected Lie group that  acts transitively on $M$ from the left,  and $K$ will always  denote the isotropy subgroup of a fixed  point $m_0$ in $M$. Then  $\fk$  will denote the Lie algebra of $K$. 

We will now assume that $\fp$ is a linear subspace in $\fg$ such that 
\begin{equation}\label{eq:3}
\fg=\fp\oplus \fk,\quad [\fp,\fk]\subseteq\fp, \quad \fk\subseteq[\fp,\fp].
\end{equation}  

 The space  $\fp$ induces a family  $\mathcal H_\fp$ of left-invariant vector fields  $X_A(g)=gA, g\in G,A\in\fp$. The associated distribution $\mathcal H_\fp(g)$, $g\in G$,
will be called horizontal, or the Ehresmann $\fp$ connection, see for instance~\cite{KMS93}. Left-invariant vector fields $gB$, $B\in\fk$ will be called vertical. Then   $\cV$  will denote the family of vertical vector fields. 

Absolutely continuous curves $g(t)\in G$ that satisfy $\frac{dg}{dt}=\dot g(t)\in\cH_\fp(g(t))$ for almost all $t$ in some interval $[0,T]$ are called {\it horizontal}. Alternatively, horizontal curves can be described as the solution curves of a control system
\begin{equation*}
\frac{dg}{dt}=\sum_{i=1}^k u_i(t)gA_i,
\end{equation*}
where $A_1,\dots,A_k$ is a basis for $\fp$, and $u_1(t),\dots,u_k(t)$ are bounded and measurable functions
on $[0,T]$.

Conditions~\eqref{eq:3} imply that $\fp+[\fp,\fp]=\fg$, which in turn implies that 
$\mathcal H_\fp$ is a two step bracket generating distribution, in the sense that  $$\mathcal H_\fp(g)+[\mathcal H_\fp,\mathcal H_\fp](g)=T_gG,\quad g\in G.$$ Therefore,   any two points in $G$ can be joined by a horizontal curve whenever $G$ is connected~\cite{AS,Chow, Rashevsky}. 

  A  curve $g(t)$ in $G$  is called  {\it a horizontal lift of a curve} $m(t)\in M=G/K$ if $g(t)$ is a horizontal curve that projects onto $m(t)$, that is, $\pi(g(t))=\phi_{g(t)}(m_0)=m(t)$. 
  \begin{proposition}
Every absolutely continuous curve $m(t)$ in $M$ is the projection of a horizontal curve. Moreover, if $g_1(t)$ and $g_2(t)$ are  horizontal lifts of a curve $m(t)$ in $M$, then $g_1(t)=g_2(t)h$ for some constant element $h\in K$. 
\end{proposition}
\begin{proof}

 Equation ~\eqref{orbit} implies that every absolutely continuous curve $m(t)$ in $M$  is the projection  of an absolutely continuous  curve $g(t)\in G$. Then   $\dot g(t)=g(t)U(t)$ for some curve $U(t)\in\fg$, and $U(t)=P(t)+Q(t)$, with $P(t)\in\fp$ and $Q(t)\in\fk$. 

Now define a new curve $\tilde{g}(t)=g(t)h(t)$, where $h(t)$  is a solution in $K$ of $\dot h(t)=-Q(t)h(t)$.  The  curve $\tilde{g}(t)$ also projects on $m(t)$ and furthermore,
 $$\dot{\tilde{g}}(t)=g(t)(P(t)+Q(t))h(t)+g(t)\dot h(t)=\tilde{g}(t)h^{-1}(t)P(t)h(t).$$
 The condition $[\fp,\fk]\subseteq\fp$ implies that $h^{-1}(t)P(t)h(t)$ is in $\fp$ for all $t$, hence $\tilde g(t)$ is a horizontal curve that  projects onto $m(t)$. 
 
 If $g_1(t)$ and $g_2(t)$ are  horizontal lifts of a curve $m(t)$ in $M$, then $\dot g_1=gU_1(t)$ and $\dot g_2=g_2U_2(t)$ for some curves $U_1(t)$ and $U_2(t)$ in $\fp$. Since $g_1(t)$ and $g_2(t)$ project onto the same  curve $m(t)$, $g_1(t)=g_2(t)h(t)$ for some curve $h(t)\in K$.  But then,
 $$g_1U_1(t)=\dot g_1=\dot g_2h+g_2\dot h=g_1(h^{-1}U_2h+h^{-1}\dot h).$$
 Hence, $U_1=h^{-1}U_2h+h^{-1}\dot h$. Since $h^{-1}U_2h$  is in $\fp$, $ h^{-1}\dot h=0$.  Therefore $h$ is constant.
 \end{proof} 
 Assume now that  $\langle.\,,.\rangle_{\fp}$  is any positive definite $\Ad_K$-invariant symmetric bilinear form on $\fp$.  This bilinear form induces a left-invariant  inner product  $\langle gA,gB\rangle_g=\langle A,B\rangle_\fp$, on $\mathcal H_\fp(g)$.   We define  $\|\frac{dg}{dt}\|_{\fp}=\sqrt{\langle g^{-1}(t)\frac{dg}{dt},  g^{-1}(t)\frac{dg}{dt}\rangle}_\fp$ for any horizontal curve $g(t)$. This metric is called a (left-invariant) {\it sub-Riemannian metric} relative to the distribution $\cH_\fp$. We denote it by the same symbol $\langle.\,,.\rangle_{\fp}$. The pair $(\cH_\fp,\langle.\,,.\rangle_{\fp})$ is called a (left-invariant) {\it sub-Riemannian structure} on $G$ and the triplet $(G,\cH_\fp,\langle.\,,.\rangle_{\fp})$ is called a {\it sub-Riemannian manifold}, see for instance~\cite{BosRos,BalTysWar}.
The sub-Riemannian metric $\langle .\,,.\rangle_{\fp}$ induces  a length $l(g,T)$ of a horizontal curve $g\colon [0,T]\to G$, by
$$
l(g,T)=\int_0^T\langle\dot g(t),\dot g(t)\rangle_{\fp}^{1/2}\,dt=\int_0^T\langle g^{-1}(t)\dot g(t),g^{-1}(t)\dot g(t)\rangle_\fp^{1/2}\,dt.
$$
The sub-Riemannian distance function $d_{sR}$ on $G$ is defined by
$$
d_{sR}(g_1,g_2)=\inf\Big\{l(g,T):\ \ \dot g(t)\in\cH_\fp(g(t)),\  \ g(0)=g_1,\ \ g(T)=g_2\Big\}.
$$
\begin{definition}\label{def:geod}
A horizontal curve $g \colon [0,T]\to G$ is called geodesic if for any $t\in (0,T)$ there exists $\varepsilon>0$ such that 
$$
d_{sR}(g(t_1),g(t_2))=\int_{t_1}^{t_2}\langle\dot g(t),\dot g(t)\rangle_{\fp}^{1/2}\,dt.
$$
for any $t_1,t_2$ with $t-\varepsilon<t_1<t_2<t+\varepsilon$.
\end{definition}

Any $\Ad_K$ left-invariant sub-Riemannian metric  $\langle .\,,.\rangle_\fp$ induces a Riemannian metric on  $M=G/K$, whereby the length of a curve $m(t)$ on an interval $[0,T]$ is equal to the length of a horizontal lift $g(t)$ of $m(t)$   in $G$. That is, the length of $m(t)$  is given by $l(m,T)=\int_0^T\sqrt{\langle U(t),U(t)\rangle}_\fp\, dt$, where $U(t)=g^{-1}(t)\dot g(t)$.
 If $\tilde g$ is another horizontal lift of $m(t)$ then,  according to the previous proposition, $\tilde g=gh$ for some constant $ h\in K$, and $\tilde U(t)=h^{-1}U(t)h$. But then $\int_0^T\sqrt{\langle U(t),U(t)\rangle}_\fp\,dt=\int_0^T\sqrt{\langle \tilde U(t),\tilde U(t)\rangle}_\fp\,dt$, by $\Ad_K$-invariance of the metric.   Hence the length of $m(t)$ is well defined. 
 \begin{definition} A Riemannian metric on $M=G/K$ that is the push forward of a sub-Riemannian left-invariant metric $\langle.\,,.\rangle_\fp$ on a horizontal distribution $\mathcal H_\fp$ in $G$ will be called {\it{homogeneous}}.
 \end{definition}

In the present paper we are essentially  interested in the structure of sub-Riemannian geodesics on $G$ and their relation to the Riemannian geodesics on $M=G/K$ relative to the homogeneous metric.  Our fundamental results  will be extracted through the length minimizing property of the geodesics and the following auxiliary optimal problem on $G$:
\begin{quotation} 
 given any two points $g_1$ and $g_2$ in $G$, find  a horizontal curve  of shortest length that connects $g_1 $ to $g_2$.  
 \end{quotation}
The solutions to this auxiliary problem are intimately related to the sub-Riemannian geodesics, because any horizontal  curve  of shortest length  that connects $g_1$ to $g_2$ is a sub-Riemannian geodesic.   Conversely,  every sub-Riemannian geodesic $g(t)$ is a curve of shortest length relative   to the points  $g_1=g(0) $ and $g_2=g(t_1)$  for    sufficiently  small $t_1$. 

This formulation permits easy comparison between the sub-Riemannian  geodesics in $G$ and the Riemannian geodesics in $M=G/K$ via the following proposition.   
\begin{proposition}\label{prop:1.2}
Suppose that $m_1$ and $m_2$ are given points in $M$. Let  $S_1=\pi^{-1}(m_1)$ and $S_2=\pi^{-1}(m_2)$ be  the fibers above these points. Then, the projection of a horizontal curve $g(t)$ is a curve of minimal length in $M$ that connects $m_1$ to $m_2$ if and only if $g(t)$ is the curve of minimal length that connects  $S_1$ to $S_2$. 
\end{proposition}
The proof  is simple and will be  omitted.


\subsection{The associated optimal control problem}\label{sec:2.2}


The optimal problem of finding curves of shortest length that connect two given points can be easily formulated as a time-optimal control  problem, but since it is more convenient to work with the energy functional $E=\frac{1}{2}\int_0^T\|U(t)\|^2_{\fp}\,dt$ rather than the length functional $\int_0^T\|U(t)\|_{\fp}\,dt$, we will, instead,   pass to  the following energy-optimal control  problem on $G$:
 \begin{quotation}
 if $T>0$ is a fixed number, and if $g_1$ and $g_2$  are fixed  points in $G$, find a horizontal curve $g(t)$ 
 that satisfies $g(0)=g_1$, $g(T)=g_2$,  along which the  total energy $E=\frac{1}{2}\int_0^T\|U(t)\|^2_{\fp}\,dt$, $U(t)=g^{-1}(t)\frac{dg}{dt}(t)$,  is minimal  among all the horizontal curves  that satisfy the boundary conditions $g(0)=g_1$  and $g(T)=g_2.$ \end{quotation}
The following proposition clarifies the relation between these two optimal control problems. 
 \begin{proposition}{\label{prop:2.1}} Every horizontal curve of shortest length is a solution of the optimal control problem for a suitable $T>0$.
 \end{proposition}
 \begin{proof}
 If $g(t)$ is any horizontal curve that is a solution of $\frac{dg}{dt}=g\,U(t)$ that connects $g_1=g(0)$ to $g_2=g(T)$, then  by the Cauchy inequality 
$$
\int_0^T\|U(t)\|_{\fp}\,dt \leq  T^{\frac{1}{2}}\Big(\int_0^T\|U(t)\|^2_{\fp}\,dt\Big)^{\frac{1}{2}}.
$$ 
The equality occurs only when $\|U(t)\|_{\fp}$ is constant, that is, when $T$ is proportional to the length of $g(t)$ on $[0,T]$.

In particular, if $g(t)$ is a curve of minimal length, and if $\|U(t)\|_{\fp}=1$, then $T$ is the length of $g(t)$ on $[0,T]$, and $E=\frac{1}{2}T$ is the minimal  value of the energy functional relative to the boundary values $g_1$ and $g_2$.
Conversely, suppose that $\hat g(t)$ is a horizontal curve parametrized by arc length: $\|\hat U(t)\|_{\fp}=1$, such that $\hat g(0)=g_1$ and $\hat g(T)=g_2$.  Then, by the above inequality, $E=\frac{1}{2}T$ is the minimal value of the energy functional over all  horizontal curves that satisfy the boundary conditions $g(0)=g_0$ and $g(T)=g_1$, and   $\hat g(t)$ attains this minimal value.  

Suppose now that  $g(s)$ is  a horizontal curve of minimal length $L$ such that  $g(0)=g_1$  and $g(S)=g_2$, for some $S>0$. Then  $L=\int_0^S\|U(\tau)\|_{\fp}\,d\tau$.
It is not difficult to show that $g(s)$ is a regular curve on the interval $[0,S]$, in the sense that $\frac{d g}{ds}(s)\neq 0$, $s\in[0,S]$.
Therefore, $g(s)$ can be reparametrized by  a parameter $s(t)$, $t\in [0,T]$ so that the reparametrized curve $\tilde g(t)=g(s(t))$ has constant speed $\|\tilde g^{-1}\frac{d\tilde g}{dt}\|_{\fp}=\lambda $. In fact, $s(t)$ is the inverse function of $t(s)=\frac{1}{\lambda}\int_0^s\|U(\tau)\|_{\fp}\,d\tau$. 
Since the  property to be horizontal and the length functional are invariant under reparametrizations,  $\tilde g(t)$ is a horizontal curve of minimal length  that reaches $g_2$  from $g_1$ in $T$ units of time, by a control of constant magnitude $\lambda=\frac{L}{T}$. Therefore, $\tilde g(t)$ is a solution of the optimal control  problem on the interval $[0,T]$ relative to the boundary  conditions $g_1$ and $g_2$. Hence it attains the optimal value  $E=\frac{1}{2}T$. It then follows that
\begin{equation*}
L=\lambda T=\int_0^T\|\tilde U(t)\|_{\fp}\,dt=T^{\frac{1}{2}}(2E)^{\frac{1}{2}}=T,
\end{equation*}
and, therefore, $\lambda=1$ and $L=T$.
\end{proof}

\begin{proposition}\label {Existence} Given any pair of points $g_1$ and $g_2$ in $G$, there exists a horizontal curve $g(t)$ such that $g(0)=g_1$, and $g(T)=g_2$, along which the total energy $E=\frac{1}{2}\int_0^T\|U(t)\|^2\,dt$ is minimal.
\end{proposition}
 The proof is essentially the same as that given  in~\cite[Proposition 9.5, p. 151]{Jc2} and will be omitted.
\begin{corollary} For any two points $g_1$ and $g_2$ in $G$ there exists  a sub-Riemannian geodesic that connects $g_1$ to $g_2$.
\end{corollary}


\subsection{The Maximum Principle and the extremal curves}\label{MP} 


We will now turn to  the Maximum Principle of Pontryagin to obtain the necessary conditions of optimality  for  the  above optimal  control problem under the additional assumptions 
that  $G$ is  semi-simple group, and  that $\langle.\,,.\rangle_\fp$ is a positive definite bilinear form  on $\fp$ that is the restriction  of a symmetric, non-degenerate  $\Ad_G$-invariant bilinear  form $\langle.\,,.\rangle$ on $\fg$.  The   $\Ad_G$-invariance  implies  that  
\begin{equation} \label{invariance}
\langle A,[B,C]\rangle=\langle B,[C,A]\rangle,
\end{equation} for all $A,B,C$ in $\fg$.  It also implies that $\langle.\,,.\rangle_\fp$ is $\Ad_K$-invariant. Typically, $\langle.\,,.\rangle$ could be  any scalar multiple of the Killing form $\Kl(A,B)=\Tr(\ad A\circ \ad B)$ that is positive definite on $\fp$.

To make an easier transition to the literature on  control theory,  we will represent curves $U(t)$ in $\fp$   in terms of an orthonormal basis $A_1,\dots,A_k$  as  $U(t)=\sum_{i=1}^k u_i(t)A_i$.  In  this representation, horizontal curves are the solutions of $\frac{dg}{dt}=\sum_{i=1}^ku_i(t)gA_i$, and their energy is given by  
$$
\frac{1}{2}\int_0^T\langle U(t),U(t)\rangle_\fp\,dt=\frac{1}{2}\int_0^T \sum_{i=1}^ku_i^2(t)\,dt.
$$

To take advantage of the left-invariant symmetries, the cotangent bundle $T^*G$   will be represented by $G\times \fg^*$, where $\fg^*$ stands for the dual of $\fg$. In this representation,  points $\xi\in T_g^*G$  are viewed as the pairs $(g,\ell)$, $\ell\in\fg^*$ defined by  $\xi (gA)=\ell(A)$ for any $A\in\fg$. Then the Hamiltonian
$h_A(\xi)$ of any left-invariant vector field $X_A(g)=gA$ is given by $h_A(\ell)=\ell(A)$.  In particular,
the control system  $ \frac{dg}{dt}=\sum_{i=1}^k u_i(t)gA_i$, together with the associated energy functional, lifts to   the extended Hamiltonian $$h_{U(t)}(\ell)=-\frac{\lambda}{2} \sum_{i=1}^k u_i^2(t)+\sum_{i=1}^ku_i(t)h_i(\ell),\quad \lambda=1,0,$$ where $h_i(\ell)=\ell(A_i)$, $i=1,\dots,k$. The corresponding Hamiltonian vector field $\vec h_{U(t)}$ is called the Hamiltonian lift of the energy-extended control system. Its integral curves $(g(t),\ell(t))$ are the solutions of 
\begin{equation*}\label{Ham}
\frac{dg}{dt}=g\, U(t),\quad \frac{d\ell}{dt}=-\ad^*U(t)(\ell(t)),\quad  U(t)=\sum_{i=1}^ku_i(t)A_i,
\end{equation*}
where $(\ad^*U(t)(\ell))(A)=\ell([U(t),A])$, $A\in\fg$, see~\cite{Jc2}.

According to the Maximum Principle, every optimal  solution $g(t)$ generated by a control  $U(t)$ is the projection of an integral curve $ (g(t),\ell(t))$ of the Hamiltonian vector field $\vec h_{U(t)}$ such that
\begin{equation}\label{Max}
h_{U(t)}(\ell(t))\geq-\frac{\lambda}{2} \sum_{i=1}^k  v_i^2+\sum_{i=1}^kv_ih_i(\ell(t)),
 \end{equation} 
 for any $(v_1,\dots,v_k)\in \mathbb R^k$ and all $t$.
 In addition, the Maximum Principle requires that $\ell(t)\neq 0$ when $\lambda=0$.
 
 Integral curves  $(g(t),\ell(t))$ of $\vec h_{U(t)}$  that satisfy the conditions of the Maximum Principle ~\eqref{Max}
 are called {\it extremal}; {\it abnormal extremal} when $\lambda=0$, and {\it normal extremal} when $\lambda=1$. In the abnormal case,  inequality~\eqref{Max} yields constraints
 \begin{equation*}
 h_i(\ell(t))=0,\quad i=1,\dots,k, 
 \end{equation*}
 while in the normal case,  the inequality shows that the  extremal control $U(t)$  is a  critical point of the Hamiltonian $h_{U(t)}(\ell)=-\frac{1}{2} \sum_{i=1}^k u_i^2(t)+\sum_{i=1}^ku_i(t)h_i(\ell)$, that is,  the extremal control is of the form $U(t)=\sum_{i=1}^kh_i(\ell(t))A_i.$

The above shows that the normal extremals are the solution curves of a single Hamiltonian vector field corresponding to the Hamiltonian
\begin{equation*}\label{HVF}
H(\ell)=\frac{1}{2}\sum_{i=1}^kh_i^2(\ell).
\end{equation*}
 We will not pursue the abnormal extremals since it is known  that their projections on $G$ cannot be optimal when $\mathcal H_\fp$ is a two step bracket generating distribution (see~\cite{AS}).
  
It follows that the sub-Riemannian geodesics are  the projections of the integral curves of the Hamiltonian vector field $\vec H$  on  energy level $H=\frac{1}{2} $, since   the energy functional is equal to the length functional  only over the horizontal curves parametrized by arc length by Proposition~\ref{prop:2.1}. That is, the sub-Riemannian geodesics are the projections of curves $(g(t),\ell(t))$ that are the solutions of 
 \begin{equation}\label{Hameq}
\frac{dg}{dt}=g\, (dH),\quad \frac{d\ell}{dt}=-\ad^*dH(\ell)(\ell(t)),
\end{equation}
on the energy level set $H=\frac{1}{2}$, where $dH=\sum_{i=1}^k h_i(\ell)A_i$ is the differential of $H$. 

For our purposes, however, it will be more convenient to express equation  ~\eqref{Hameq}  on the tangent bundle $G\times\fg$   rather than  the cotangent bundle $G\times\fg^*$.  For that reason,
 $\fg^*$ will be identified with $\fg$ via the bilinear form $\langle .\,,.\rangle$, i.e., 
 \begin{equation*}\label{eq:444}
 \ell\in\fg^*\Longleftrightarrow L\in\fg\quad\text{if and only if}\quad\ell(A)=\langle L,A\rangle,\ \text{for all} \ A\in\fg.
 \end{equation*} 
 Then, relying on~\eqref{invariance},
$$
\begin{array}{lcl}\langle \frac{dL}{dt},A\rangle&=&\frac{d\ell}{dt}(A)=-\ad^*dH(\ell)(\ell(t))(A)=-\ell([dH,A])\\
&&\\
&=&\langle L,[dH,A]\rangle=\langle[dH,L],A\rangle .
\end{array}$$ 
Since $A$ is an arbitrary element of $\fg$, $\frac{dL}{dt}=[dH,L]$. Hence the extremal equations ~\eqref{Hameq}  are equivalent to 
 \begin{equation}\label{Hameq1}
\frac{dg}{dt}=g(t)dH(\ell(t)),\qquad \frac{dL}{dt}=[dH(\ell),L].
\end{equation}

Let now $\fp^\perp$ denote the orthogonal complement of $\fp$ in $\fg$  relative to  $\langle.\,,.\rangle$. Since $\langle.\,,.\rangle$ is a symmetric and non-degenerate quadratic form, $\fg=\fp\oplus\fp^\perp.$  Then each $L\in\fg$ can be written as $L=L_\fp+L_{\fp^\perp}$ with $L_\fp\in\fp$ and $L_{\fp^\perp}\in\fp^\perp$. 
Relative to the orthonormal basis $A_1,\dots,A_k$, $L_\fp=\sum_{i=1}^k\langle  L,A_i\rangle A_i$. But, $\langle L,A_i\rangle=\ell(A_i)=h_i(\ell)$, hence $dH=L_\fp$. 

The above shows that the Hamiltonian $H$  can be written as $H=\frac{1}{2}\langle L_\fp,L_\fp\rangle$ and the associated equations~\eqref{Hameq1}  can be written as
\begin{equation*}\label{Hameq2}
\frac{dg}{dt}=g\, L_\fp,\quad \frac{dL}{dt}=[L_\fp,L].
\end{equation*}

Under the previous assumptions, together with the condition that $\fp ^\perp$ is a Lie subalgebra of $\fg$, we now come to the main theorem of the paper. 
\begin{theorem}\label{prop:4.3}  Assume that $\fp^\perp$ is a Lie subalgebra of $\fg$. Then, the sub-Riemannian geodesics are given by  
\begin{equation}\label{eq:50}
g(t)= g(0)\exp(t(P_\fp+P_{\fp^\perp}))\exp(-tP_{\fp^\perp}),
\end{equation} 
for some constant elements $P_\fp\in\fp$  and $P_{\fp^\perp}\in\fp^\perp$ with $\|P_{\fp}\|=1$. 

The Riemannian geodesics  on $M=G/K$ are the projections of the sub-Riemannian geodesics for which $P_\fp+P_{\fp^\perp}$ is orthogonal to $\fk$.
\end{theorem}

\begin{proof}
 Sub-Riemannian geodesics are the projections of
\begin{equation*}\label{eq:48}
\frac{dg}{dt}=g(t)L_\fp(t),\qquad\frac{dL}{dt}=[L_\fp(t),L(t)],
\end{equation*}
on the energy level set $H=\frac{1}{2}\|L_\fp\|^2=\frac{1}{2}$. 

We now address the solutions of 
\begin{equation*}
\frac{dL}{dt}=[L_\fp(t),L(t)],
\end{equation*} 
or, equivalently, the solutions of 
\begin{equation}\label{eq:49}
\frac{dL_\fp}{dt}+\frac{dL_{\fp^\perp}}{dt}=[L_\fp(t),L(t)]=[L_\fp(t),L_{\fp^\perp} (t)].
\end{equation}

Due to the $\ad_\fk$-invariance property of $\langle .\, ,. \rangle$, and the assumption that $\fp^\perp$ is a Lie subalgebra of $\fg$, we have
$$\langle [L_\fp,L_{\fp^\perp}],\fp^\perp\rangle=\langle L_\fp,[L_{\fp^\perp},\fp^\perp]\rangle=0,$$
 from which we conclude that 
 $[L_\fp (t),L_{\fp^\perp}(t)] \in \fp$.  Consequently, $\displaystyle\frac{dL_{\fp^\perp}}{dt}=0$, hence $L_{\fp^\perp}(t)=P_{\fp^\perp},$ for some constant element $P_{\fp^\perp}\in\fp^\perp$.
 
 So, equation  ~\eqref{eq:49} reduces to   
$$
\frac{dL_{\fp^\perp}}{dt}=0,\qquad \frac{dL_\fp}{dt}=[L_\fp (t), P_{\fp^\perp}].
$$
Therefore,
\begin{equation}\label{**}
U(t)=L_\fp(t)=\exp(tP_{\fp^\perp})P_{\fp}\exp(-tP_{\fp^\perp}),\end{equation}
 where $P_{\fp}=L_\fp(0)$.  

The corresponding sub-Riemannian geodesics are the solutions of $\frac{dg}{dt}=g(t)U(t)$. In order to show that the solution of this differential equation  has the required form, we define
$\tilde g(t)=g(t)\exp(tP_{\fp^\perp})$ and use~\eqref{**} to obtain 
$$
\begin{array}{lcl}
\frac{d\tilde{g}}{dt}&=&g(t)U(t)\exp(tP_{\fp^\perp})+\tilde g(t) P_{\fp^\perp}\\
&&\\
&=&\tilde {g}(t)\big(\exp(-tP_{\fp^\perp})U(t)\exp(tP_{\fp^\perp})+P_{\fp^\perp}\big)\\
&&\\
&=&\tilde g(t)(P_\fp+P_{\fp^\perp}).
\end{array}$$
Hence, $g(t)$ is given by~\eqref{eq:50}.

To complete the proof, we will use the fact  that the geodesics in the quotient space $M=G/K$  are the projections of the extremal curves  in $G$ that satisfy the transversality conditions, implied by Proposition~\ref{prop:1.2}, which, after the identification $\ell\rightarrow L$, means that $L(t)$, the extremal curve that projects onto a geodesic in $G/K$, is orthogonal to $\fk$ at $t=0$ and $t=T$. Since the horizontal distribution $\cH _\fp$ is $K$-invariant, the Hamiltonian lift $h_{A}(L)=\langle L,A\rangle$ of any left-invariant vector field $gA$, $A\in\fk$  is constant along the solutions of~\eqref{eq:49}.  Indeed, $\frac{d}{dt}\langle L(t),A\rangle=\langle [L_{\fp(t)^\perp},L_\fp(t)],A\rangle=\langle L_{\fp(t)^\perp},[L_\fp(t),A]\rangle=0$, because of $[L_\fp(t),A]\in\fp$. 

 So, since $\langle L(t),A\rangle$ is constant, $\langle L(0),A\rangle=0$ if and only if $\langle L(t),A\rangle=0$ for all $t$. It follows that the orthogonality conditions reduce to $L(0)=P_\fp+P_{\fp^\perp}\perp\fk$.
\end{proof}

As a corollary we have the following  result,~\cite[Proposition 8.30]{Jc2}.
\begin{corollary} \label{cor1}
If $\fp^\perp=\fk$, the sub-Riemannian geodesics are given by  
\begin{equation}\label{eq:50a}
g(t)= g(0)\exp(t(P_\fp+P_{\fk}))\exp(-tP_{\fk}),
\end{equation} 
for some constant elements $P_\fp\in\fp$  and $P_{\fk}\in\fk$ with $\|P_{\fp}\|=1$. 

The Riemannian geodesics  on $M=G/K$ are the projections of the sub-Riemannian geodesics~\eqref{eq:50a} for which $P_{\fk}=0$.\end{corollary} 

\begin{proposition}\label{prop:4.4} If $\fp^\perp=\fk$, then the projection of a sub-Riemannian geodesic on the quotient space $M=G/K$ is a curve of constant geodesic curvature relative to the  Riemannian  metric. 
\end{proposition} 

\begin{proof} Recall 
 that the geodesic curvature of a curve $m(t)$ parametrized by its arc length is equal to the length of the covariant derivative of $\frac{dm}{dt}$ along $m(t)$. In this context it is most convenient to express the covariant derivative in terms of its horizontal lift as follows.

 Let $Y(t)$ denote a vector field defined along the curve $m(t)$ in $M$. Then $Y(t)$ is the projection of a horizontal curve $g(t)W(t)$, for some curve $W(t)\in\fp$, where $g(t)$ denotes the horizontal curve that projects onto $m(t)$. It follows that $\frac{dm}{dt}$ is the projection of a curve $g(t)U(t)$ for some $U(t)\in\fp$. Then the covariant derivative $\frac{DY}{dt}$ of $Y(t)$ along $m(t)$ is the projection of 
 \begin{equation*}\label{eq:56A}
 g(t)\Big(\frac{dW}{dt}+\frac{1}{2}[U(t),W(t)]_{\fp}\Big),
 \end{equation*}
on $M$, where $[U,W]_\fp$ denotes the projection of $[U,W]$ on $\fp$.

In this situation $m(t)=\pi(g(t))$, where $g(t)$ is given by (\ref{eq:50a}).
Then, $\frac{dg}{dt}=g(t)U(t)=g(t)(\exp(tP_{\fk})P_\fp\exp(-tP_{\fk}))$ and $\|U(t)\|=\|P_\fp\|=1$. Hence $m(t)$ is parametrized by arc length.  
It follows that 
$$
\frac{D}{dt} \Big(\frac{dm}{dt}\Big)=\pi_*\Big(g(t)\frac{dU}{dt}\Big)=\pi_*\big(g(t) \exp(tP_{\fk})[P_\fp,P_{\fk}]\exp(-tP_{\fk})\big).$$
But then $\Big\|\frac{D}{dt}\Big( \frac{dm}{dt}\Big)\Big\|=\big\|  \exp(tP_{\fk})[P_\fp,P_{\fk}]\exp(-tP_{\fk })\big\|=\big\|\,[P_\fp,P_{\fk} ]\,\big\|$.

\end{proof}


\section{Homogeneous metrics on general Stiefel manifolds $\St^n_k(V)$}\label{sec:Grassmann}

\subsection{Stiefel manifolds}
A Stiefel manifold $\St^n_k(V)$  consists  of ordered $k$ orthonormal vectors $v_1,\dots,v_k$ in an $n$-dimensional Euclidean vector space~$V$.  We will focus on the cases when the vector space $V$ is an Euclidean space $V=\IR^n$, an Hermitian complex vector space $V=\IC^n$, or a quaternionic space $V=\IH^n$ (with the right multiplication by scalars) equipped with its inner product $(u,v)=\sum_{l=1}^n\bar u_l v_l$.  We will identify points  $m=(v_1,\dots,v_k)\in \St^n_k(V)$  with matrices  $M_{nk}$ whose columns  consist of the coordinate vectors $v_1,\dots,v_k$ with respect to a chosen orthonormal basis for $V$. Each such matrix  $M_{nk}$ satisfies $M^*_{nk}M_{nk}=I_k$, where $^*$ stands for the  transpose, the complex conjugate, or the quaternion conjugate, depending on the case.

The quaternionic Stiefel manifold requires some additional explanations. Let $1,\ii,\jj,\kk$ denote the standard basis in the quaternion algebra $\IH$. Then, every quaternion is a linear combination $q=q_0+q_1\ii+q_2\jj+q_3\kk$ with real coefficients. The conjugate $\bar q$ of the quaternion $q$ is given by $\bar q=q_0-q_1\ii-q_2\jj-q_3\kk$. The product of two quaternions is defined by using the coordinate representation and the law $\ii^2=\jj^2=\kk^2=\ii\jj\kk=-1$.
It follows that $\overline{qq'}=\bar q'\bar q$ and $\bar qq=q_0^2+q_1^2+q_2^2+q_3^2=|q|^2$.  
The notion of conjugacy  extends to $\IH^n$ and yields  the inner product $(v,w)=\sum_{l=1}^n\bar v_lw_l$.  It readily follows that  
\begin{equation}\label{eq:quat-inner-product}
(v\alpha,w)=\bar\alpha(v,w),\quad(v,w\alpha)=(v,w)\alpha,\quad(v,w)=\overline{(w,v)},
\end{equation}
for any $\alpha\in\IH$.
Then $\Sp(n)$  is the group of matrices that leave invariant the quaternionic Hermitian product~\eqref{eq:quat-inner-product}. It follows that  $\Sp(n)$ consists of  $n\times n$ matrices $\Theta$ with quaternionic entries that satisfy $\Theta^*\Theta=\Theta\Theta^*=I_n$.
  Reminiscent of the unitary group, one can show that   $\Sp(n)$ is isomorphic to  $\Sp(2n,\IC)\cap \U(2n)$~\cite[page 445]{Hl}.

\subsection{Stiefel manifolds as homogeneous spaces}


All Stiefel manifolds are homogeneous manifolds, and can be realized as the quotients of Lie groups through several group actions. Below we will describe two such actions. To avoid unnecessary repetitions,  we will use  $G_n$  to denote    $\SO(n)$ in the real case, $\SU(n)$ in the complex case, and $\Sp(n)$ in the quaternionic case. 

Both $G_n$ and $G_{k}$ act  on elements of $\St_k^n(V)$, represented as $n\times k$ matrices. The first group acts by the matrix multiplication on the left, and the second group by the matrix multiplication on the right.

Let us first consider the full group action.

$\bullet $ {\it $\St_k^n(V)$ as a homogeneous manifold $G_n\times G_k/G_k\times G_{n-k}$}. 
Here the  full group $G=G_n\times G_k$ acts on $\St_k^n(V)$  by  
\begin{equation*}\label{eq:13}
\phi((r,s),m)=rms^{-1},\quad m\in \St_k^n(V),\quad r\in G_n,\quad s\in G_k.
\end{equation*}
The action is transitive, and the orbit  through  $m=I_{nk}=\begin{pmatrix} I_k\\0\end{pmatrix}$  consists of the  matrices  $M_{nk}=rI_{nk}s^{-1}$, $r\in G_n$, $s\in G_k$. The isotropy group $K$, that leaves $I_{nk}$ fixed, consists of   matrices $r\in G_n$ and $s\in G_k$ such that 
\begin{equation}\label{eq:15}
r I_{nk}=I_{nk}s.
\end{equation}
The matrix $r\in G_n$ can be written in block form as $r=\begin{pmatrix} R_1&R_2\\R_3&R_4\end{pmatrix}$, with $R_1$ an $k\times k$-matrix, and the remaining matrices of the corresponding sizes.
 Then~\eqref{eq:15} holds for $r$ only when
$R_1=s$ and $R_2=R_3=0$. Therefore, $r=\begin{pmatrix} R_1&0\\0&R_4\end{pmatrix}$. This shows that the isotropy group $K$ is   
\begin{equation}\label{eq:isotr-big}
K=\Big\{\Big(\begin{pmatrix} S&0\\0&T\end{pmatrix},S\Big):\ S\in G_k,\ T\in G_{n-k}\Big\}  \cong G_k\times G_{n-k}. 
\end{equation} 

$\bullet $ {\it $\St_k^n(V)$ as a homogeneous manifold $G_n/G_{n-k}$}. 
The  reduced  group action  is given by 
\begin{equation*}\label{eq:14}
\phi(g,m)=gm,\quad m\in \St_k^n(V),\quad g\in G_n.
\end{equation*}
The orbit of the reduced action through the matrix $m=I_{nk}$ consists of  matrices  $M_{nk}=gI_{nk}\in \St_k^n(V)$, $g\in G_n$, i.e., the first $k$ columns of $g$. The isotropy group  at the point $I_{nk}$ is equal to 
\begin{equation}\label{eq:isotr-reduced}
K=\Big\{\begin{pmatrix} I_k&0\\0&H\end{pmatrix}:\ H\in G_{n-k}\Big\}\cong G_{n-k}.
\end{equation}

The above shows that each Stiefel manifold $\St_k^n(V)$ can be represented in two ways: as the full quotient $(G_n\times G_k)/G_k\times G_{n-k}$, as well as the reduced quotient $G_n/G_{n-k}$.

In what follows we will  write $\St_k^n(V)=G/K$ with the understanding that  $G$ and $K$ are either $G=G_n$ and $K=G_{n-k}$, or $G=G_n\times G_k$ and $K=G_k\times G_{n-k}$, depending on the context. Then, $\fg$ and  $\fk$ will denote the Lie algebras of $G$ and $K$. Similarly $\fg_k$, $\fg_{n-k}$ will denote the Lie algebras of $G_k$ and $G_{n-k}$. 


\subsubsection{Homogeneous metrics}


Each Lie algebra $\fg$ is endowed with a positive definite bilinear form $\langle.\,,.\rangle$. If $\fg_n$ is either $\fso(n)$ or $\fu(n)$, then the form is given by 
\begin{equation}\label{eq:17a}
\langle  A,B\rangle=-\frac{1}{2}\Tr(AB)\quad\text{for}\quad A,B\in\fg_n,
\end{equation}
with $\Tr(A)$ the trace of the matrix $A$. On $\fg_n=\fsp(n)$ the form is written as
\begin{equation}\label{eq:17b}
\langle A,B\rangle=-\frac{1}{ 4}\Tr\big(AB+(AB)^*\big).
\end{equation}
We will often refer to the above forms on $\fg$  as {\it the trace form}, and   to the induced metric as the {\it trace metric}. The trace form extends to  the product $\fg_k\times\fg_{n-k}$ with $\langle.\,,.\rangle=\langle.\,,.\rangle_1+\langle.\,,.\rangle_2$, where $\langle.\,,.\rangle_1$ is the trace form in $\fg_k$ and $\langle.\,,.\rangle_2$ is the trace form in $\fg_{n-k}$.
The trace form on $\fg$ is $\Ad_G$ invariant and satisfies~\eqref{invariance}. Thus the corresponding left-invariant Riemannian metric $\langle .\,,.\rangle_G$ is also bi-invariant.  

We will now introduce three decompositions $\fg=\fp\oplus\fk$ that conform to ~\eqref{eq:3} and induce the left-invariant horizontal distributions $\mathcal H_{\fp}$ that are relevant for the applications. 
\\

 {\it The reduced horizontal distribution} is the horizontal distribution associated with the representation $\St_k^n(V)=G_n/G_{n-k}$. It is induced by the orthogonal complement
 \begin{equation}\label{eq:19b}
\fp=\Big\{\begin{pmatrix} A&B\\-B^*&0\end{pmatrix}:\,  A\in \fg_k \Big\}
\end{equation}
with respect to the trace metric to the isotropy algebra
\begin{equation}\label{eq:16b}
\fk=\Big\{\begin{pmatrix} 0&0\\0&D\end{pmatrix} :\, D\in\fg_{n-k}\Big\}
\end{equation}
of the isotropy group $K$ in~\eqref{eq:isotr-reduced}.
\\

The other two horizontal distributions are associated with the representation  $\St_k^n(V)=(G_n\times G_{n-k})/(G_k\times G_{n-k})$.

 First, is {\it the orthogonal horizontal distribution on $G=G_n\times G_{n-k}$}.
It is induced by the orthogonal complement
\begin{equation}\label{eq:19a}
\fp=\fk^\perp=\Big\{\Big(\begin{pmatrix} A&B\\-B^*&0\end{pmatrix},-A\Big)  :\,  A\in \fg_k \Big\}
\end{equation} 
with respect to the trace metric to the isotropy algebra
\begin{equation}\label{eq:16a}
\fk=\Big\{\Big(\begin{pmatrix} C&0\\0&D\end{pmatrix},C\Big):\ C\in \fg_k,\ D\in\fg_{n-k}\Big\}
\end{equation}
of the isotropy group $K$ in~\eqref{eq:isotr-big}.

Second,
{\it the quasi-geodesic horizontal distribution on $G=G_n\times G_{n-k}$.}
This horizontal distribution is induced by the vector space 
\begin{equation}\label{eq:new}
\fp=\Big\{\Big(\begin{pmatrix} 0&B\\-B^*&0\end{pmatrix},A\Big) :\  A\in \fg_k\Big\} 
\end{equation} 
that is not orthogonal to $\fk$ in~\eqref{eq:16a}.

 Evidently $\fp$  is transversal to $\fk$, and satisfies $\fg=\fp\oplus \fk$, because for any $\begin{pmatrix} A_1&A_2\\-A_2^*&A_3\end{pmatrix}\in \fg_n$  and $B\in \fg_{k}$  we have 
$$
\Big(\begin{pmatrix} A_1&A_2\\-A_2^*&A_3\end{pmatrix},B\Big)
=\Big(\begin{pmatrix}0&A_2\\-A_2^*&0\end{pmatrix},B-A_1\Big)
+\Big(\begin{pmatrix} A_1&0\\0&A_3\end{pmatrix},A_1\Big).
$$
The distribution defined by the left translations of  $\fp$  in~\eqref{eq:new} is called {\it quasi-geodesic  distribution } to emphasize  the connection with curves called {\it{quasi-geodesic}}, whose significance was demonstrated  in~\cite{KL}  in  interpolation problems arising from applications. 

\begin{lemma}\label{lem:pk-decomposition}
In all three cases above, the reduced, orthogonal, and quasi-geodesic  case, the decomposition $\fg=\fp\oplus\fk$ satisfies~\eqref{eq:3}.
\end{lemma}
\begin{proof}
In the reduced and the orthogonal cases
$\langle [\fp,\fk],\fk\rangle=\langle \fp,[\fk,\fk]\rangle=0,$ because $\fk$ is a Lie algebra, and it  is orthogonal to $\fp$. Therefore, $[\fp,\fk]\subset  \fp$. It is also true for the quasi-geodesic case by a direct calculation.  So in all cases $[\fp,\fk]\subseteq\fp$.  

We will now show that $\fk\subseteq[\fp,\fp]$. In the calculations below, $E_{i,j}$ denotes the $n\times n$ matrix with entry $(i,j)$ equal to $1$ and all other entries equal to~$0$, and $A_{i,j}= E_{i,j}-E_{j,i}$.  Matrices $A_{i,j}$ satisfy the following commutator properties: 
\begin{equation*}\label{commutator-properties}
\left[A_{i,j}, A_{f,l}\right]=-\delta_{il}A_{j,f}-\delta_{jf}A_{i,l}+\delta_{if}A_{j,l}+\delta_{jl}A_{i,f},
\end{equation*}
where $\delta_{ij}$ denotes the Kronecker delta function. 

Taking into consideration the structure of the matrices in $\fk$ and $\fp$ in the reduced case,  given respectively by (\ref{eq:16b}) and (\ref{eq:19b}), it is clear that
$$
\{ A_{k+i,k+j}, \,\, 1\leq i< j \leq n-k\}
$$
is a basis for $\fk$, while 
$$
\{ A_{i,j}, \,\, 1\leq i< j \leq k\} \cup \{ A_{i,k+j}, \,\, 1\leq i\leq k,\, 1\leq j \leq n-k\}
$$
is a basis for $\fp$. Since for any $A_{k+i,k+j} \in \fk$, there exists $A_{l,k+i}, A_{l,k+j}\in \fp$ such that 
\begin{equation} \label{new1}
A_{k+i,k+j}=\left[ A_{l,k+i}, A_{l,k+j} \right], 
\end{equation}
we have proved that in the reduced case $\fk\subseteq[\fp,\fp]$. To show that this inclusion is also true for the other two distributions, it is enough to take into account the structure of the matrices that define the subspaces $\fk$ (given by (\ref{eq:16a})) and $\fp$, given either by~\eqref{eq:19a} or~\eqref{eq:new}, and use~\eqref{new1} together with the following extra identity:
\begin{equation*} \label{new2}
A_{i,j}=\left[ A_{i,k+l}, A_{j,k+l} \right],
\end{equation*}
 which is valid for all $ 1\leq i < j \leq k$ and $l\in \{1, \cdots , n-k\}$. 
 
%
%

\end{proof}

Thus we obtain the following corollary.

\begin{corollary}\label{cor:bracket-generating}
All three horizontal distributions $\mathcal H_\fp$ are  two-step bracket generating distributions, $\mathcal H_\fp(g)+[\mathcal H_\fp,\mathcal H_\fp](g)=T_gG$ for all $g\in G$.
\end{corollary}

\begin{theorem}\label{th:isom} 
The homogeneous metrics on the Stiefel manifold $\St_k^n(V)$ induced by the sub-Riemannian metrics relative to the reduced horizontal  distribution on $G_n$, and   the sub-Riemannian  metric  on $G_n\times G_{k}$ relative to the quasi-geodesic distribution are equal, and they are different from the    
homogeneous metric induced by the sub-Riemannian metric relative to  the orthogonal horizontal distribution on $G_n\times G_{k}$.
\end{theorem}

\begin{proof}
We start  with the proof of the first statement. Let $\dot m (t)$ denote the tangent vector of a curve $m(t)$ in $\St_k^n(V)$. Then, $m(t)=g(t)I_{nk}$ and $m(t)=r(t)I_{nk}s^*(t)$ for some horizontal curves $g(t)\in G_n$ and $(r(t),s(t))$ in $G=G_n\times G_k$.
Then  
\begin{equation}\label{eq:*}
\dot m (t)=g(t)\, W(t)I_{nk}=r(t)(U_1(t)I_{nk}-I_{nk}U_2(t))s^*(t),
\end{equation}
where 
$$
\frac{dg}{dt}=g(t)W(t),\qquad\frac{dr}{dt}=r(t)U_1(t),\qquad\frac{ds}{dt}=s(t)U_2(t).
$$
It will be convenient  to embed $G_k$ into $G_n$ by identifying $s\in G_k$ with $\begin{pmatrix}S&0\\0&I_{n-k}\end{pmatrix}$, and   identify $I_{nk}$ with  $\begin{pmatrix}I_k&0\\0&0\end{pmatrix}$, so that all the matrices above can be written
as  $n\times n$ matrices  in block form:  
$$W=\begin{pmatrix} A&B\\-B^*&0\end{pmatrix},\quad 
U_1=\begin{pmatrix} 0&C\\-C^*&0\end{pmatrix},\quad
U_2\sim \begin{pmatrix} D& 0\\0&0\end{pmatrix},\ \ A, D \in\fg_k.
$$
Let $\|\dot m(t)\|_1$  and $\|\dot m (t)\|_2$ denote the lengths of $\dot m(t)$ relative to the homogeneous metrics induced by the reduced and the quasi-geodesic  distributions.
Since the sub-Riemannian metrics on $G$ are left-invariant, we need only to compare the norms of the horizontal vectors on the corresponding Lie algebras. Thus
$$
\|W(t)\|^2_{\fp}=\|A(t)\|^2+\Tr(B(t)B^*(t)),
$$
is equal to  $\|\dot m(t)\|_1^2$, and
$$
\|(U_1(t),U_2(t))\|^2_{\fp}=\|U_1(t)\|^2+\|U_2(t)\|^2=\|D(t)\|^2+\Tr(C(t)C^*(t))
$$
is equal to $\|\dot m(t)\|_2^2$. We need to show that $\|\dot m(t)\|_1=\|\dot m(t)\|_2$.

Since $I_{nk}$ commutes with $s^*$,  
\begin{equation*}\label{eq:relation-grs}
m(t)=g(t)\, I_{nk}=r(t)I_{nk}s^*(t)=r(t)s^*(t)I_{nk},
\end{equation*}
 and therefore,
 $r(t)s^{*}(t)h=g(t)$ for some constant $h\in K$, where $K$ is given by~\eqref{eq:isotr-reduced}.   
Now     equation~\eqref{eq:*}
implies that $rs^*hWh^*I_{nk}=r(U_1-U_2)I_{nk}s^*$, or $
(U_1-U_2)I_{nk}=s^*hWh^*sI_{nk},
$
 because $I_{nk}$  commutes with $U_2$, and $I_{nk}h=hI_{nk}=I_{nk}$.   This equality implies that 
 \begin{equation*}\label{U_1-U_2}
U_1-U_2=\begin{pmatrix}- D(t)&C(t)\\-C^*(t)&0\end{pmatrix}=\tilde W,
\end{equation*}
because both of these matrices are of the form $\begin{pmatrix} X&Y\\-X^*&0\end{pmatrix}$. Here we denote $\tilde W=s^*hWh^*s$.  
Then,
\begin{eqnarray*}
\|\dot m(t)\|^2_1&=&
  \|W(t)\|^2_{\fp}
=\|\tilde W(t)\|^2_{\fp}=-\frac{1}{2}\Tr(\tilde W^2(t))
\\
&=&-\frac{1}{2}\Tr\begin{pmatrix} -D(t)&C(t)\\
-C^*(t)&0\end{pmatrix}^2
=
\|D(t)\|^2+\Tr(C(t)C^*(t))
\\
&=&\displaystyle \|(U_1(t),U_2(t))\|^2_{\fp}=\|\dot m(t)\|^2_2.
\end{eqnarray*} 

Now we prove the second statement of the proposition.  In this case 
$$\frac{dg}{dt}=g(t)\, \tilde{U}_1(t),\quad \frac{ds}{dt}=s(t)\, \tilde{U}_2(t)$$ with   $\tilde{U}_1(t)=\begin{pmatrix} A(t)&B(t)\\-B^*(t)&0\end{pmatrix}$, $\tilde{U}_2(t)=\begin{pmatrix} -A(t)&0\\0&0\end{pmatrix}$. The norm of $(\tilde{U}_1, \tilde{U}_2)$ with respect to orthogonal horizontal distribution is given by 
$$
\|(\tilde{U}_1(t), \tilde{U}_2(t))\|^2_{\fp}=\|\tilde{U}_1(t)\|^2+\|\tilde{U}_2(t)\|^2=2\|A(t)\|^2+\Tr(B(t)B^*(t)).
$$  
A calculation similar to the one above shows that 
\begin{equation*}\label{U_1-U_2a}
\tilde W=s^*hWh^*s=\begin{pmatrix} 2A(t)&B(t)\\-B^*(t)&0\end{pmatrix}= \tilde{U}_1 - \tilde{U}_2.
\end{equation*}
Therefore,

$$
\|W(t)\|^2_{\fp}=\|\tilde W(t)\|^2_{\fp}=4\|A(t)\|^2+\Tr(B(t)B^*(t))\neq \|(\tilde U_1(t), \tilde U_2(t))\|_{\fp}^2.
$$

\end{proof}


\subsection{Sub-Riemannian geodesics}~\label{sec:subRiemanian1}
 
\subsubsection{Geodesics on $\St_k^n(V)$, induced by the reduced horizontal distribution}\label{sec:revisited}
\vspace*{0,3 cm}
In this case, $\fp=\fk^\perp$  where  $\fk$ and $\fp$ are given by (\ref{eq:16b}) and (\ref{eq:19b}) respectively, and  hence by    (\ref{eq:50a}), the sub-Riemannian geodesics on $G_n$ are  of the form
\begin{equation*}\label{eq:51}
g(t)=g_0\exp\Big( t\begin{pmatrix} A&B\\-B^*&D\end{pmatrix} \Big)\begin{pmatrix} I_k&0\\0&e^{-tD}\end{pmatrix},
\end{equation*} 
Their  projections on $\St_k^n(V)$ are given by  
$$
m(t)=\pi(g(t))=g_0\exp \Big( t\begin{pmatrix} A&B\\-B^*&D\end{pmatrix} \Big) I_{nk},
$$
and the Riemannian geodesics are of  the form
\begin{equation}\label{eq:52a}
m(t)
=g_0\, e^{t\Omega}I_{nk},\quad \Omega=\begin{pmatrix} A&B\\-B^*&0\end{pmatrix}.\end{equation}  

These geodesics  are called canonical in~\cite{EAS} and normal in ~\cite{FJ}.

 
\subsubsection{Geodesics on $\St_k^n(V)$ induced by the orthogonal horizontal distribution}\label{subsec:orth_distr}

 
In this case $\fp$ and $\fk$ are given by~\eqref{eq:19a} and~\eqref{eq:16a}, respectively.
Then $P_{\fp}=\Big(\begin{pmatrix} A&B\\-B^*&0\end{pmatrix},-A\Big)$, and $P_{\fk}=\Big(\begin{pmatrix} C&0\\0&D\end{pmatrix},C\Big)$, which leads to 
\begin{eqnarray*}\label{eq:53}
g(t)
&=&
g_0\exp\big(t(P_{\fp}+P_{\fk})\big)\exp(-tP_{\fk})\nonumber
\\
&=&
g_0\Big(\exp\Big(t\begin{pmatrix} A+C&B\\-B^*&D\end{pmatrix}\Big),e^{t(-A+C)}\Big)\Big(\begin{pmatrix} e^{-tC}&0\\0&e^{-tD}\end{pmatrix},e^{-tC}\Big).
\end{eqnarray*}
The projection is given by
\begin{eqnarray*}
\pi(g(t))
&=&
g_0\Big(\exp \Big(t\begin{pmatrix} A+C&B\\-B^*&D\end{pmatrix}\Big)
\Big(\begin{pmatrix} e^{-tC}&0\\0&e^{-tD}\end{pmatrix}I_{nk}e^{tC}e^{t(A-C)}\Big)
\\
&=&
g_0\Big(\exp \Big(t\begin{pmatrix} A+C&B\\-B^*&D\end{pmatrix}\Big)I_{nk}e^{t(A-C)}\Big).
\end{eqnarray*}
If $g_0=(r,s)\in G_n\times G_k$, then
the geodesics through the point $g_0I_{nk}=rI_{nk}s^*$, corresponding to  $P_{\fk}=0$, i.e., $C=D=0$, have the form
\begin{equation}\label{eq:54}
m(t)=g_0\exp\Big(t\begin{pmatrix} A&B\\-B^*&0\end{pmatrix}\Big)I_{nk}e^{tA}
=re^{t\Omega}I_{nk}e^{tA}s^*
\end{equation} 
 
\subsubsection{Geodesics on $\St_k^n(V)$ induced by quasi-geodesic horizontal distribution}\label{subsec:qg_distr}


The quasi-geodesic distribution is generated by~\eqref{eq:16a} and~\eqref{eq:new}. 
An easy calculation shows that  $\fp^\perp=\Big\{\Big(\begin{pmatrix} E&0\\0&F\end{pmatrix},0\Big): E\in\fg_k, F\in\fg_{n-k}\Big\}.$
Evidently, $\fp^\perp$ is a Lie subalgebra of $\fg$, hence its sub-Riemannian geodesics are given by~\eqref{eq:50}. Let 
\begin{equation} \label{new}
P_{\fp}=\Big(\begin{pmatrix} 0&B\\-B^*&0\end{pmatrix},A\Big)\in\fp\quad\text{and}\quad P_{\fp^{\perp}}=\Big(\begin{pmatrix} E&0\\0&F\end{pmatrix},0\Big)\in\fp^\perp.
\end{equation}
Then 
\begin{eqnarray}\label{eq:sub_riemQG}
g(t)
&=&
g_0\, \exp(t(P_{\fp}+P_{\fp^{\perp}}))\exp(-tP_{\fp^{\perp}})\nonumber
\\
&=&
g_0\, \Big(\exp{ \Big( t\begin{pmatrix} E&B\\-B^*&F\end{pmatrix}\Big) }\begin{pmatrix}e^{-tE}&0\\0&e^{-tF}\end{pmatrix},e^{tA}\Big).
\end{eqnarray}
 If $g_0=(r,s)\in G_n\times G_k$, then 
\begin{eqnarray}\label{eq:55}
\pi(g(t))
&=&
r\Big(\exp{\Big(t\begin{pmatrix} E&B\\-B^*&F\end{pmatrix}\Big)}\begin{pmatrix} e^{-tE}&0\\0&e^{-tF}\end{pmatrix} I_{nk}e^{-tA}s^*\nonumber
\\
&=&
r\Big(\exp{ \Big( t\begin{pmatrix} E&B\\-B^*&F\end{pmatrix}\Big) }\begin{pmatrix}  e^{-tE}e^{-tA}&0\\0&e^{-tF}\end{pmatrix} I_{nk}s^*. 
\end{eqnarray}

 According to Theorem~\ref{prop:4.3}, the  geodesics on $\St_n^k(V)$ are the projections of the above curves  for which $P_{\fp}+P_{\fp^{\perp}}$ is orthogonal to $\fk$.  The orthogonal complement $\fk^\perp$ consists of matrices of the form  $\Big(\begin{pmatrix} X&Y\\-Y^*&0\end{pmatrix},-X\Big)$. So, if  $P_{\fp}$ and $P_{\fp^{\perp}}$ are as in~\eqref{new},
then $P_{\fp}+P_{\fp^{\perp}}\in\fk^\perp$ if and only if  
 $E=-A$ and $F=0$. In such a case, 
 $$
 P_{\fp}=\Big(\begin{pmatrix} 0&B\\-B^*&0\end{pmatrix},A\Big),\qquad 
 P_{\fp^{\perp}}=\Big(\begin{pmatrix} -A&0\\0&0\end{pmatrix},0\Big),
$$
and 
$$
P_{\fp}+P_{\fp^{\perp}}=\Big(\begin{pmatrix} -A&B\\-B^*&0\end{pmatrix},A\Big).
$$ 
It then follows  from ~\eqref{eq:55} that  the geodesics  at $g_0=(r,s)\in G_n\times G_k$ are of the form
\begin{equation}\label{eq:geod_qg}
m(t)=re^{t\widetilde\Omega}I_{nk}s^*,  \quad  \widetilde \Omega=\begin{pmatrix} -A&B\\-B^*&0\end{pmatrix}.
\end{equation}
At first glance, formulas~\eqref{eq:geod_qg} and~\eqref{eq:52a} give different curves. To show that it is not the case, let 
 $g=rs^*$, where, whenever convenient, we identify $s\in G_k$ with $s=\begin{pmatrix}S&0\\0&I_{n-k}\end{pmatrix}\in G_n$.  Then, 
$$
m(t)=re^{t\widetilde\Omega}I_{nk}s^*=re^{t\widetilde\Omega}s^*I_{nk}=rs^*se^{t\widetilde\Omega}s^*I_{nk}=g\, e^{t\Omega}I_{nk},
$$
where $\Omega=s \widetilde\Omega s^*=\begin{pmatrix}-SAS^*&SB\\-B^*S^*&0\end{pmatrix}$. 

We are now almost ready to relate the above formalism to  the quasi-geodesic curves. The following lemma will lead the way.

\begin{lemma}\label{lemma:relations}
Curves 
\begin{equation*}m(t)=
r\Big(\exp{t\begin{pmatrix} E&B\\-B^*&F\end{pmatrix}}\begin{pmatrix}  e^{-tE}e^{-tA}&0\\0&e^{-tF}\end{pmatrix} I_{nk}s^*
\end{equation*}
 in $\St_k^n(V)$  that are  the projections of sub-Riemannian geodesic on $G_n\times G_k$ relative to the quasi-geodesic distribution are also the projections of  horizontal curves $g(t)$  in $G_n$  that are  solutions of 
\begin{equation}\label{eq:horizontal-curve}
\frac{dg}{dt}=g(t)\begin{pmatrix}-\tilde A&e^{t\tilde A}e^{t\tilde E} \tilde Be^{-t  F}\\ -e^{t  F} \tilde B^*e^{-t\tilde E}e^{-t\tilde A}&0\end{pmatrix} 
\end{equation}
with $\tilde A=SAS^*$, $\tilde E=SES^*$,  and $\tilde B=SB$.

\end{lemma}
\begin{proof}
The sub-Riemannian geodesic relative to the quasi-geodesic distribution is given by~\eqref{eq:sub_riemQG} and its projection on the Stiefel manifold is given by~\eqref{eq:55}. The latter  can be written as 
\begin{equation}\label{eq:srQG-R}
m(t)=g_0e^{t\Phi}\Delta(t)I_{nk},\quad g_0=rs^*,
\end{equation}
with
$$
e^{t\Phi}=s\exp\Big(t\begin{pmatrix}E&B\\-B^*&F\end{pmatrix}\Big)s^*=\exp\Big(t\begin{pmatrix}\tilde E&\tilde B\\-\tilde B^*& F \end{pmatrix}\Big)\quad\text{and}\quad
$$
$$
\Delta=s\begin{pmatrix}e^{-tE}e^{-tA}&0\\0&e^{-tF}\end{pmatrix}s^*=\begin{pmatrix}e^{-t\tilde E}e^{-t\tilde A}&0\\0& e^{-t  F}\end{pmatrix}.
$$
The curve $ g(t)= g_0e^{t\Phi}\Delta(t)$ is the curve on $G_n$ having the derivative
$$
\dot{ g}(t)= g_0e^{t\Phi}\Delta\Big(\Delta^{-1}\Phi\Delta+\Delta^{-1}\dot\Delta\Big)=g(t)\Big(\Delta^{-1}\Phi\Delta+\Delta^{-1}\dot\Delta\Big).
$$
  A straightforward calculation shows that 
\begin{equation*}
\Delta^{-1}\Phi\Delta+\Delta^{-1}\dot\Delta=\begin{pmatrix}-\tilde A&e^{t\tilde A}e^{t\tilde E}\tilde Be^{-t  F}\\e^{-t F}\tilde B^*e^{-t\tilde E}e^{-t\tilde A}&0\end{pmatrix}.
\end{equation*}
Therefore, $ g(t)$ is a horizontal curve in $G_n$ satisfying the conditions of the lemma.
\end{proof}

We finally come to the quasi-geodesic curves. \begin{definition}\label{def:5.1}  Quasi-geodesic curves through a point $m=rI_{nk}s^*$ in $ \St^n_k(V)$ are curves $\gamma (t)$  having the form   $\gamma (t)=r\exp(t\Psi)I_{nk}e^{-tA}s^*$ for some  matrices $\Psi=\begin{pmatrix} 0&B\\-B^*&0\end{pmatrix}$, with $B\in \mathcal M_{k(n-k)}(V)$ and  $A\in\fg_k$. 
\end{definition}

  Alternatively, quasi-geodesic curves  can be defined as  curves 
  \begin{equation*}\label{def-quasi2}
  \gamma (t)=\exp(tX)m\exp(tY),
  \end{equation*} 
  where 
$ X=r\Psi r^*$, and  $Y=-sAs^*\in\fg_k$. Indeed, 
\begin{eqnarray*}
\gamma (t)&=&r\exp(t\Psi)I_{nk}\exp(-tA)s^*=r\exp(t\Psi)r^* rI_{nk}s^*s\exp(-tA)s^*
\\
&=& \exp{(t(r\Psi r^*))}rI_{nk}s^*\exp(-tsAs^*) =\exp(tX)m\exp(tY).
\end{eqnarray*}

\begin{proposition}\label{prop:5.1} Quasi-geodesic curves  coincide with  the projections of  sub-Riemann\-ian geodesics  in~\eqref{eq:sub_riemQG} with $P_{\fp^{\perp}}=0$.  They  are curves of constant geodesic curvature. A quasi-geodesic is a Riemannian geodesic on $\St_k^n(V)$ if either $P_{\fp}=(0,A)$, or $P_{\fp}=\Big(\begin{pmatrix} 0&B\\-B^*&0\end{pmatrix},0\Big)$.
\end{proposition} 

\begin{proof} The first statement is a consequence of formula~\eqref{eq:55}. 

To show the second statement we will use Lemma ~\ref{lemma:relations}. When $P_{\fp^\perp}=0$,   $E=F=0$, and $m(t)=g(t)I_{nk}$, where $g(t)=g_0\exp({t\begin{pmatrix}0&\tilde B\\-\tilde B^*&0\end{pmatrix}})\begin{pmatrix}e^{-t\tilde A}&0\\0&I\end{pmatrix}$. Then equation ~\eqref{eq:horizontal-curve} reduces to \begin{equation*}\label{eq:horizontal-is-geodesic}
 \frac{dg}{dt}= g(t)\begin{pmatrix} -\tilde A&e^{t\tilde A}\tilde B\\-\tilde B^*e^{-t\tilde A}&0\end{pmatrix}=g(t)\begin{pmatrix}e^{t\tilde A}&0\\0&I\end{pmatrix}\begin{pmatrix}-\tilde A&\tilde B\\-\tilde B^*&0\end{pmatrix}\begin{pmatrix}e^{-t\tilde A}&0\\0&I\end{pmatrix}.
 \end{equation*}
Since $m(t)$ is the projection of a sub-Riemannian geodesic, it is parametrized by the arc length.
That  implies that $U(t)=\begin{pmatrix}e^{t\tilde A}&0\\0&I\end{pmatrix}\begin{pmatrix}-\tilde A&\tilde B\\-\tilde B^*&0\end{pmatrix}\begin{pmatrix}e^{-t\tilde A}&0\\0&I\end{pmatrix}$ is of unit length.
  
The  geodesic curvature of $m(t)$ is given by $\|\frac{D}{dt}(\frac{dm}{dt})\|$ relative to the homogeneous metric,  where $\frac{D}{dt}$ denotes the covariant derivative. An argument completely analogous to that in 
 Proposition~\ref{prop:4.4} shows that
  \begin{equation*}\label{curv}
 \Big\|\frac{D}{dt}(\frac{dm}{dt})\Big\|=\Big\|\frac{dU}{dt}\Big\|_{\fp}=\Big\|\Big[\begin{pmatrix}A&0\\0&0\end{pmatrix},\begin{pmatrix}-A&B\\-B^*&0\end{pmatrix}\Big]\Big\|=\Big\|\begin{pmatrix}0&-AB\\-B^*A&0\end{pmatrix}\Big\|
 \end{equation*}
 Evidently,  $\|\frac{D}{dt}(\frac{dm}{dt})\|=0$ if and only if either $A=0$, or $B=0$.\end{proof}

The last two statements  of Proposition \ref{prop:5.1} were  proved earlier in ~\cite{KL} by  direct computations without any recourse to Lie groups.


\subsection{The ambient (Euclidean, Hermitian, or quaternion Hermitian) metric on the Stiefel manifolds}

Each Stiefel manifold $\St_k^n(V)$ is  a closed subset of 
 the vector space  $\mathcal M_{nk}(V)$ of $n\times k$ matrices with the entries in $V$,  $V=\mathbb R^n$,  $V=\mathbb C^n$, or $V=\mathbb H^n$, endowed with  the usual quadratic form $ \langle A,B\rangle_{\mathcal M}=\Tr(A^*B)$, $A,B\in \mathcal M_{nk}(V)$.   We will  refer to $\mathcal M_{nk}(V) $  together with the metric induced  by $\langle.\,,.\rangle_\mathcal M$ as the {\it ambient manifold}. 
 
  Thus each Stiefel manifold $\St_k^n(V)$ is identified with  a closed submanifold defined by $ \{X\in \mathcal M_{nk}(V):\,  X^*X=I_n\}$  of $\mathcal M_{nk}$. And, consequently, its   tangent space $T_X\St_k^n(V)$ is identified with
$$
T_X\St_k^n(V)=\{\dot X\in \mathcal M_{n,k}: X^*\dot X=-\dot X^*X\}, \quad X\in \St_k^n(V).
$$
  We will now consider $\St_k^n(V)$ as a Riemannian manifold with the metric  given by 
\begin{equation*}\label{eq:Riem-metric}
(\dot X_1,\dot X_2)=\langle \dot X_1,\dot X_2\rangle_{\mathcal M}=\Tr(\dot X_1^*\dot X_2)
\end{equation*}
This choice of a metric will be called {\it{ambient}}.

We will now show that  the ambient metric can be lifted to a metric  $(.\,,.)_\fp$ on the space $\fp$ given by~\eqref{eq:19b}, which then induces  yet another left-invariant sub-Riemannian structure $(\mathcal H_p,(.\,,.)_\fp)$ on $G_n$. We will then extract the Riemannian geodesics relative to the ambient metric  by  analyzing the sub-Riemannian geodesics induced by $(.\,,.)_\fp$ by a procedure that is analogous to the one described in Section~\ref{sec:subRiemanian1}.

We have shown earlier that every tangent vector $\dot X$ at a point  $X$  can be lifted to a unique horizontal vector $gU$ above $X$, that is  $\dot X=gUI_{nk}$, where $U=\begin{pmatrix}A&B\\-B^*&0\end{pmatrix}$ for  suitable matrices $A$ and $B$. Now,
\begin{eqnarray*}\label{eq:induced metric}
 \Tr(\dot X_1^*\dot X_2)&=&\Tr(I_{nk}^*U_1^*g^*gU_2I_{nk})
=
\Tr( I_{nk}^*U_1^*U_2I_{nk})
\\
&=&\Tr\Big(\begin{pmatrix}A_1^*&-B_1\end{pmatrix}\begin{pmatrix} A_2&\\-B^*_2\end{pmatrix}\Big)
=
\Tr(A^*_1A_2)+\Tr(B_1B_2^*).\nonumber
\end{eqnarray*}
Thus the ambient bilinear form  lifts to
\begin{equation}\label{eq:euclidean on p}
(U_1,U_2)_{\fp}=\Tr(A_1^* A_2)+\Tr(B_1B_2^*)
\end{equation}
for $U_1=\begin{pmatrix}A_1&B_1\\-B_1^*&0\end{pmatrix}\in\fp$, and $U_2=\begin{pmatrix}A_2&B_2\\-B_2^*&0\end{pmatrix}\in\fp$.
We therefore have two bilinear  forms on $\fp$,  the trace form $\langle.\,,.\rangle_\fp$ given by ~\eqref{eq:17a} or~\eqref{eq:17b}, and the quadratic form $(.\,,.)_\fp$ given by~\eqref{eq:euclidean on p}. They are related by the formula 
\begin{equation*}\label{amb}
(U_1,U_2)_\fp=\langle U_1, DU_2+U_2D\rangle_\fp,
\end{equation*}
where $D=\begin{pmatrix} I_k&0\\0&0\end{pmatrix}$.

\subsubsection{Sub-Riemanian problem on $(G_n,\mathcal H_p,(.\,,.)_\fp)$}
We will now obtain the sub-Riemann\-ian geodesics  associated with  minimizing  $\frac{1}{2}\int_0^T(g^{-1}\frac{dg}{dt},g^{-1}\frac{dg}{dt})_\fp \,dt$ over the horizontal cur\-ves that satisfy $g(0)=g_1$ and $g(T)=g_2$. 

The Hamiltonian equations, based on the Maximum Principle,  will be obtained much in the same manner as in  Section~\ref{MP}. For the moment we assume that $V=\IR^n$ or $V=\IC^n$. We let $L$ denote the vector in $\fg$ dual to some $l\in\fg^*$ with respect to the metric $\langle .\,,.\rangle$ defined in~\eqref{eq:17a}, and write $L=L_\fp+L_\fk$ for the decomposition relative to the factors $\fp$ and $\fk$ defined in~\eqref{eq:19b} and~\eqref{eq:16b}. It follows that the regular extremals are the projections from $T^*G$ onto $G$ of the integral curves $L(t)$ of the lifted Hamiltonian
$$h_U(L)(t)=-\frac{1}{2}(U(t),U(t))_{\fp}+\langle L_\fp,U\rangle_\fp,$$
subject to the optimality condition, that the extremal control $ U(t)$ and  the associated dual vector $L(t)$  maximize $h_U(L(t))$ over all controls $U$ in $\fp$.
If $ L(t)=\begin{pmatrix} A(t)&B(t)\\-B^*(t)&C(t)\end{pmatrix}$,  and $U=\begin{pmatrix}u&v\\-v^*&0\end{pmatrix}$, then
\begin{equation*}\label{eq:25a}
h_U(L(t))=-\frac{1}{2}\Tr(u^*u)-\frac{1}{2}\Tr(vv^*)-\frac{1}{2}\Tr(A(t)u)+\frac{1}{2}\Tr(B^*(t)v+B(t)v^*).
\end{equation*}
 It follows that $h_U(L(t))$ attains the maximum relative to the control functions precisely when  $2u(t)=A(t)$, and $v(t)=B(t)$. Therefore, the extremal curves $(g(t),L(t))$ are the integral curves  of the Hamiltonian system generated by
\begin{equation}\label{eq:hamiltonianH}
H=\frac{1}{2}(U,U)_\fp=\frac{1}{4}\|A\|^2+\frac{1}{2}\Tr(BB^*),
\end{equation}
i.e., they are the solutions of the system
$
\frac{dg}{dt}=g(t)U(t)$, $\frac{dL}{dt}=[U(t),L(t)]$.
where
$dH=\begin{pmatrix} \frac{1}{2}A&B\\-B^*&0\end{pmatrix}=U$.
Hence,
$$
\begin{pmatrix} \dot {A}&\dot {B}\\-\dot {B}^*&\dot {C}\end{pmatrix}=\begin{pmatrix} 0&\frac{1}{2}AB-BC\\ \frac{1}{2}B^*A-CB^*&0\end{pmatrix}.
$$
It follows that $A$ and $C$ are constant and that $B(t)=\exp(\frac{t}{2}A)B(0)\exp(-tC)$. This yields
\begin{equation*}
U(t)
=\begin{pmatrix} \frac{1}{2}A&e^{\frac{1}{2}tA}B(0)e^{-tC})\\ -e^{tC}B^*(0)e^{-\frac{1}{2}tA}&0\end{pmatrix}
=
e^{tQ}\begin{pmatrix} \frac{1}{2}A&B(0)\\-B^*(0)&0\end{pmatrix} e^{-tQ},
\end{equation*}
where  $Q=\begin{pmatrix} \frac{1}{2}A&0\\0&C\end{pmatrix}$.

The extremal curve $g(t)$ in $G_n$  is a solution of $\frac{dg}{dt}=g(t)(e^{tQ}Pe^{-tQ})$, with $P=\begin{pmatrix} \frac{1}{2}A&B(0)\\-B^*(0)&0\end{pmatrix}$ and $Q=\begin{pmatrix} \frac{1}{2}A&0\\0&C\end{pmatrix}$. It then follows that the sub-Riemannian geodesic
\begin{equation*}\label{eq16}
g(t)=g_0e^{t(P+Q)}e^{-tQ}
\end{equation*}
projects on the Stiefel manifold as
\begin{equation}\label{eq:27a}
X(t)=\pi (g(t))=g_0\, e^{t(P+Q)}e^{-tQ}I_{nk}=g_0e^{t(P+Q)}I_{nk}e^{-\frac{t}{2}A},\quad A\in\fg_k,
\end{equation}
with $P+Q=\begin{pmatrix} A&B(0)\\-B^*(0)&C\end{pmatrix}$.  

For the Riemannian geodesics  on the Stiefel manifold we set $C=0$, because of the transversality conditions.

Let us now show that $X(t)=g_0\, e^{t(P+Q)}e^{-tQ}I_{nk}$ with the term $P+Q=\begin{pmatrix} A&B\\-B^*&0\end{pmatrix}\in\fp$ and $Q=\begin{pmatrix}\frac{A}{2}&0\\0&0\end{pmatrix}$   satisfies the Euler-Lagrange equation
\begin{equation}\label{eq:28a}
\ddot X +X\dot X^* \dot X=0\quad \Longleftrightarrow\quad \dot X=Y,\quad \dot Y=-X(Y^*Y),
\end{equation} 
 found   in~\cite{EAS} and ~\cite{FJ} 
for the case $V=\mathbb R^n$.

We have $X(t)=g(t) I_{nk}$,  $\dot X(t)=Y(t)=g(t)U(t)I_{nk}$, where
$g(t)=g_0\, e^{t(P+Q)}e^{-tQ}$, and $U(t)=e^{tQ}\begin{pmatrix} \frac{1}{2}A&B\\-B^*&0\end{pmatrix} e^{-tQ}$.
 Then,
\begin{eqnarray*}
\dot Y
&=&
g\, U^2I_{nk}+g\dot{U}I_{nk}=g\, (e^{tQ}P^2e^{-tQ})I_{nk}+g\, (e^{tQ}[P,Q]e^{-tQ})I_{nk}
\\
&=&
g\, e^{tQ}(P^2+[P,Q])e^{-tQ}I_{nk}
\\
&=&
g\begin{pmatrix}  e^{t\frac{A}{2}}(\frac{1}{4}A^2-BB^*)e^{-t\frac{A}{2}}&0\\0&0\end{pmatrix} =g\begin{pmatrix} \frac{1}{4}A^2-e^{t\frac{A}{2}}BB^*e^{-t\frac{A}{2}}&0\\0&0\end{pmatrix}.
\end{eqnarray*}
On the other hand,
\begin{eqnarray*}
X(Y^*Y)
=-gI_{nk}U^2I_{nk}
=
-g\begin{pmatrix} \frac{1}{4}A^2-e^{t\frac{A}{2}}BB^*e^{-t\frac{A}{2})}&0\\0&0\end{pmatrix}.
\end{eqnarray*}
Therefore, $X(t)$ in ~\eqref{eq:27a} satisfies the Euler-Lagrange  equation~\eqref{eq:28a} when $C=0$.

The calculations in the case $V=\IH^n$ are similar. We obtain
\begin{eqnarray*}
h_U(L(t))&=&-\frac{1}{2}\Tr(u^*u)-\frac{1}{2}\Tr(vv^*)-\frac{1}{4}\Tr(Au+(Au)^*)
\\
&+&\frac{1}{4}\Tr(Bv^*+(Bv^*)^*+B^*v+(B^*v)^*).
\end{eqnarray*}
The maximum is achieved at $2u=A$ and $v=B$ giving the Hamiltonian~\eqref{eq:hamiltonianH} for the corresponding metric. The rest of the calculations, identical to the ones above, show that the quaternionic geodesics are given by  ~\eqref{eq:27a}, with $C=0$.

Observe now that all the homogeneous metrics discussed  above coalesce into a single metric in the extreme  cases $k=n$ and $k=1$, and in both cases agree with the ambient metric. This is  obvious in the case that  $k=n$, for  then $\St_n^n(V)$ is equal to $G_n$, and the homogeneous metric is equal to the bi-invariant metric on $G_n$.  

In the case $k=1$, the Stiefel manifolds $\St_1^n(V)$ is the unit sphere,  $S^{n-1}$ in the real case, $S^{2n-1}$ in the complex case, and $S^{4n-1}$ in the quaternionic case. 
To see that the homogeneous metric coincides with the metric inherited from the ambient space $V$, note that  the matrix $\Omega=\begin{pmatrix}A&B\\-B^*&0\end{pmatrix}$ in ~\eqref{eq:52a} and~\eqref{eq:54} is equal to  $\begin{pmatrix} 0&b\\-b^*&0\end{pmatrix}$  where $b$ is a row  vector  when $k=1$. Hence 
\begin{equation*}
m(t)=\exp\Big(t\begin{pmatrix}0&b\\-b^*&0\end{pmatrix}\Big)e_1=(I\cos{\|b\|t}+\frac{1}{\|b\|}\begin{pmatrix}0&b\\-b^*&0\end{pmatrix}\sin{\|b\|t})e_1.
\end{equation*} 
Therefore, $m(t)$ is a solution of 
\begin{equation*}
\ddot m(t)+\|b\|^2m(t)=0.
\end{equation*}
It may be somewhat surprising that  in all other cases, $1<k<n$,  the metric on $\St_k^n(V)$  inherited from the ambient space $\mathcal M_{nk}(V)$ is  less natural than  the homogeneous metric on $\St_k^n(V)$ relative to the reduced action of $G_n$.



\section{Grassmann manifolds $\Gr_k^n(V)$}


 We will   now demonstrate  the relevance of the  sub-Riemannian structures on Lie groups, as described in the first part of this paper, to   the canonical Riemannian structure of the Grassmann manifolds $\Gr^n_k(V)$.  We will also make use of the fact that $\St_k^n(V)$ is a principal $G_k$ bundle over $\Gr_k^n(V)$ to examine the geometric properties of the projections   to  $\Gr_k^n(V)$ of  the   sub-Riemannian geodesics  in $\St_k^n(V)$.
 
Recall that $\Gr_k^n(V)$ is the set of all $k$-dimensional vector subspaces of an $n$-dimensional vector space  $V$. We will continue with our notations from above, with  $V$ one of $\IR^n,\IC^n$, or $\IH^n$ endowed with its usual metric,   except that for the moment  $G_n$  will denote $\Orth(n)$ in the real case, $\U(n)$ in the complex case, rather than $\SO(n)$ and $\SU(n)$ as before, while  in the quaternionic  case $G_n$ will be $\Sp(n)$, the same as before. 
  
Then $\Gr_k^n(V)$ can be embedded into $G_n$ by identifying  each vector space  $W$ in $\Gr_k^n(V)$  with the orthogonal reflection $R_W$  defined by
$$R_W(x)=
\begin{cases}
x\quad &\text{if} \quad x\in W,
\\
-x, \quad &\text{if} \quad x\in W^\perp.
\end{cases}
$$

Group $G_n$ acts on Grassmann manifolds $\Gr_k^n(V)$ under the action
$$
(\mathcal O,W)\rightarrow \mathcal OW=\{\mathcal Ow:\ w\in W\},\quad \mathcal O\in G_n.
$$
The action of $G_n$ on $\Gr_k^n(V)$ can be also expressed in terms of the reflections $R_W$ by the following:
\begin{equation}\label{eq:projRW}
(\mathcal O, R_W ) \to \mathcal OR_W \mathcal O^*,\quad \mathcal O\in G_n.
\end{equation}
It is easy to verify that this action is transitive. Therefore, $\Gr_k^n(V)$ can be realized as the quotient $G_n/K$, where 
\begin{equation}\label{eq:Knew}
K=\Big\{\begin{pmatrix} A&0\\0&C\end{pmatrix},\  A\in G_k,\ C\in G_{n-k}\Big\}\ \cong\  G_k\times G_{n-k}
\end{equation}
is the isotropy group of $R_{W_0}=\begin{pmatrix} I_k&0\\0&-I_{n-k}\end{pmatrix}$ associated with the vector space $W_0$ spanned by the standard vectors $e_1, \ldots, e_k$.

For our purposes it is desirable to work with connected Lie groups. So, from now on we assume that $G$ and $K$ are connected, that is $G$ is equal to $\SO(n)$, $\SU(n)$ or $\Sp(n)$, $K$  is modified accordingly, and the quotient $G/K$ is the oriented Grassmannians instead. 

Alternatively, the decomposition  $\fg=\fk\oplus\fp$ could have been obtained through the involutive automorphism $\sigma (g)=DgD^{-1}$ where we denote $D=\begin{pmatrix}I_k&0\\0&-I_{n-k}\end{pmatrix}$. Then $K$ is equal to the subgroup of fixed points of $\sigma$: $K=\{\sigma(g)=g:\, g\in G\}$. Note that $D$ can be also seen as the orthogonal reflexion $R_{W_0}$ across  $W_0$  the linear span of $e_1,\dots,e_k$.

In general, an involutive automorphism $\sigma\neq \Id$ on a Lie group $G$ is an automorphism that satisfies $\sigma^2=\Id$. It follows that the tangent map $\sigma_*$ at the group identity is a Lie algebra automorphism that satisfies  $\sigma_*^2=\Id$. Hence $(\sigma_*-\Id)(\sigma_*+\Id)=0$, and therefore, $\fg=\fk\oplus\fp$, where 
$$
\fp=\{A\in\fg:\ \sigma_*(A)=-A\},\quad\text{and}\quad \fk=\{A\in\fg:\ \sigma_*(A)=A\}.
$$ 
The subspaces $\fp$ and $\fk$ are orthogonal relative to the Killing form and satisfy  Cartan relations
\begin{equation*}\label{Cartan}\fg=\fp\oplus \fk,\quad
[\fp,\fk]\subseteq \fp,\quad [\fp,\fp]\subseteq\fk,\quad[\fk,\fk]\subseteq\fk.
\end{equation*}
On semisimple Lie algebras $[\fk,\fp]=\fp$, and on simple Lie algebras $[\fp,\fp]=\fk$, and therefore
$\fp+[\fp,\fp]=\fg$, see~\cite{Jc2}.

Our case here is a particular case of this general situation  since the trace form is a scalar multiple of the Killing form. Moreover, $[\fp,\fp]=\fk$, as can be easily verified.  So we are in the situation  where $\fp+[\fp,\fp]=\fg$.

 Therefore, the left-invariant distribution $\mathcal H_{\fp}$ with values in $\fp$ defines a natural sub-Riemannian problem on $G_n$: 
\begin{quotation} 
Find the sub-Riemannian geodesics on $G_n$ and identify those that project on the  Riemannian geodesics in the Grassmannian $\Gr_k^n(V)$.
\end{quotation} 

According to Theorem \ref{prop:4.3}, the sub-Riemannian geodesics are  given by
\begin{equation*}
g(t)=g_0\, e^{t(P_{\fp}+P_{\fk})}e^{-tP_{\fk}}\label{subgeo}
=g_0\exp\Big(t\begin{pmatrix} A&B\\-B^*&C\end{pmatrix}\Big)\begin{pmatrix} e^{-tA}&0\\0&e^{-tC}\end{pmatrix},
\end{equation*}
and     their  projections on $\Gr_k^n(V)$, obtained by ~\eqref{eq:projRW}, are of the form
\begin{equation}\label{subpr}
R(t)=g_0\exp\Big(t\begin{pmatrix} A&B\\-B^*&C\end{pmatrix}\Big)D\exp\Big(-t\begin{pmatrix} A&B\\-B^*&C\end{pmatrix}\Big)g_0^*,\end{equation}

Since $\fp ^\perp =\fk$, the curves in~\eqref{subpr}  have constant geodesic curvature  in $\Gr_k^n(V)$  by~Proposition  \ref{prop:4.4}.  The Riemannian geodesics on $\Gr_k^n(V)$ are given by  Corollary~\ref{cor1},
\begin{equation}\label{eq:566}
R(t)
=g_0\exp\Big(t\begin{pmatrix} 0&B\\-B^*&0\end{pmatrix}\Big)\, D\, \exp\Big(-t\begin{pmatrix} 0&B\\-B^*&0\end{pmatrix}\Big)g_0^*.
\end{equation}

Equation~\eqref{eq:566} can be expressed in the form
\begin{equation*}\label{eq:57}
R(t)=\exp{(tP)}R_0\exp{(-tP)},
\end{equation*}
where $P=g_0\begin{pmatrix}0&B\\-B^*&0\end{pmatrix}g_0^*$.  It is easy to verify that $R_0P+PR_0=0$.
The converse is also true: if $P\in \fg_n$ satisfies $R_0P+PR_0=0$, then $P=g_0\begin{pmatrix}0&B\\-B^*&0\end{pmatrix}g_0^*$.


Let us now note that the involutive automorphism $\sigma(g)=DgD^{-1}$ is an isometry for the above sub-Riemannian  structure on $G$ since 
$$
\langle \sigma_*(A),\sigma_*(B)\rangle_{\fp}=\langle-A,-B\rangle_{\fp}=\langle A,B\rangle_{\fp}.
$$ 
We will presently show that this isometry  accounts for  the geodesic symmetry of the Riemannian Grassmannian manifolds. 

To elaborate, first note that $\sigma (e^{tA})=e^{t\sigma_*(A)}$ for any $A\in\fg$.  Next, let  $F_g\colon G\rightarrow G$ be the mapping defined  by $F_g(h)=g\,\sigma(g^{-1}h)$ at each $g\in G$.  It follows that  $F_g$ is an isometry for the sub-Riemannian structure and satisfies $F_g(g)=g$.  If  
$g_0\, e^{t(P_{\fp}+P_{\fk})}e^{-tP_{\fk}}$ is a sub-Riemannian geodesic at $g_0$ then
$$F_{g_0}(g(t))=g_0\, \sigma(e^{t(P_{\fp}+P_{\fk})}e^{-tP_{\fk}})=g_0\, e^{t\sigma_*(P_{\fp}+P_{\fk})}e^{-t\sigma_*P_{\fk}})=g_0\, e^{t(-P_{\fp}+P_{\fk})}e^{-tP_{\fk}}.
$$
It follows that $F_g$ maps the sub-Riemannian geodesics at $g$ onto the  sub-Riemannian geodesics at $g$. The sub-Riemannian geodesics that project onto the Riemannian geodesics  are given by $P_\fk=0$, and we have $$F_g(g\, e^{tP_\fp})=g\, e^{-tP_\fp}.$$
 It follows that $S_{\pi (g_0)}(\pi (g))=\pi\circ F_{g_0}(g)$ is an  isometry that satisfies  
 $S_p(\gamma(t))=\gamma(-t)$  for any geodesic curve $\gamma (t)$ with $\gamma(0)=p$. 
 
 Any Riemannian space $M$ in which the map $S_p\colon \gamma(t)\mapsto\gamma(-t)$  is an isometry  for any geodesic $\gamma$ is called {\it{symmetric Riemannian space}}~\cite{Eb,Hl}. The  above  shows that the oriented Grassmannian manifolds $\Gr_k^n(V)$  with the above homogeneous metric  belongs to  the class of symmetric Riemannian spaces.

\subsection{Relation between Stiefel and Grassmann manifolds}

  Every $k$-dimensional subspace $W$ of $V$ is in one to one correspondence with the orthogonal reflection $R_W$ and the orthogonal projection  
 $\Pi_W$  defined by
$$
\Pi_W(x)=
\begin{cases}
x,\quad&\text{if}\quad x\in W
\\
0,\quad&\text{if}\quad x\in W^\perp
\end{cases}.
$$
The map
$$
W\rightarrow \Pi_W\in\{ A\in \fgl(V):\ A^*=A,\ A^2=A,\ \dim(\ker(A))=n-k\}
$$
defines a matrix representation of $\Gr_k^n(V)$ in terms of the orthogonal projections.  
The passage from  the reflections  to the  projections  is given  by a simple formula
\begin{equation*}\label{eq:20}
R_W=2\Pi_W-I.
\end{equation*}

Therefore $R_W=gR_{W_0}g^*$ corresponds to $ 2\Pi_W-I=g(2\Pi_{W_0}-I)g^*$, or $\Pi_W=g\Pi_{W_0}g^*$. In particular, $\Pi_{W_0}=\begin{pmatrix} I_k&0\\0&0\end{pmatrix}$ when $R_{W_0}=D$.
In this representation the geodesic equations ~\eqref{eq:566} become
$\Pi(t)=e^{tP}\Pi_0e^{-tP}$, where $\Pi_0P+P\Pi_0=P$.

There is a natural projection from the Stiefel manifold to the Grassmann manifold, because
every point $q = [v_1,\ldots,v_k]$ in $\St^n_k(V)$  can be projected to the vector space $W$ spanned by $v_1,\dots,v_k$.  In terms of the orthogonal projections, the projection $\Pi(q)=W$ is  given by $\Pi(q) =\sum_{i=1}^kv_i \otimes v_i^*$. When $q$ is regarded as an $n\times (n-k)$ matrix with  columns $v_1,\dots,v_k$, then  $\Pi(q) =\sum_{i=1}^kv_i \otimes v_i^*=qq^*$. This identification then yields
\begin{equation*}\label{FF^*}
\Pi(q) = qq^*\in \Gr_k^n(V), \quad q\in \St^n_k(V). 
\end{equation*}
Evidently, $\Pi^{-1}(\Pi(q)) = \{qh,\, h\in G_k\}$. Therefore, $\Pi$ is a surjection, and $\St^n_k(V)$ is a principal $G_k$ bundle over $\Gr_k^n(V)$ relative to the action of $h\in G_k$ given by $\phi(h,q)=qh$, see for instance~\cite{AutMar,Mt}. 

Let us now go back to the curves on $\St_k^n(V )$ that are the projections of various sub-Riemannian geodesics.  Equations  ~\eqref{eq:srQG-R} capture all these curves. They are of the form 
\begin{equation}\label{eq:Subproj}
q(t)=g_0e^{t\Phi}\Delta(t)I_{nk},\quad g_0=rs^*,\end{equation}
with
$$
e^{t\Phi}=s\exp\Big(t\begin{pmatrix}E&B\\-B^*&F\end{pmatrix}\Big)s^*=\exp\Big(t\begin{pmatrix}\tilde E&\tilde B\\-\tilde B^*& F\end{pmatrix}\Big),\quad\text{and}\quad
$$
$$
\Delta=s\begin{pmatrix}e^{-tE}e^{-tA}&0\\0&e^{-tF}\end{pmatrix}s^*=\begin{pmatrix}e^{-t\tilde E}e^{-t\tilde A}&0\\0& e^{-t  F}\end{pmatrix},
$$
where $\tilde A=SAS^*$, $\tilde E=SES^*$, $\tilde B=SB$, and $s=\begin{pmatrix}S&0\\0& I_{n-k}\end{pmatrix}$.

The curves in~\eqref{eq:Subproj} for arbitrary $\Phi$ correspond to the projection of sub-Riemannian curves relative to the quasi-geodesic distribution. The case $ E= F=0$ corresponds to the orthogonal distribution, and  $A=E=F=0$ corresponds to the projection of  the sub-Riemannian geodesics relative to the reduced orthogonal distribution.   
The projection of  equations (\ref{eq:Subproj}) on the Grassmannians is given by
\begin{equation*}
\Pi(q(t))=g_0\, e^{t\tilde\Phi}\begin{pmatrix}I_k&0\\0&0\end{pmatrix}e^{-t\tilde\Phi}g_0^*, \quad 
\tilde \Phi=\begin{pmatrix}\tilde E&\tilde B\\-\tilde B^*& F\end{pmatrix},
\end{equation*}
or in terms of the orthogonal reflections, by
\begin{equation}\label{subrmpr}
R(t)=g_0\, e^{t\tilde\Phi}\begin{pmatrix}I_k&0\\0&-I_{n-k}\end{pmatrix}e^{-t\tilde\Phi}g_0^*.
\end{equation}
\begin{proposition}\label{constcurv} The projections of  sub-Riemannian geodesics on $\St_k^n(V)$ project onto  curves of constant curvature in $\Gr_k^n(V)$. Their curvature is zero precisely  when $E=F=0$. In particular, the   quasi-geodesic curves project onto Riemannian geodesics in $\Gr_k^n(V)$.
\end{proposition}
\begin{proof} Equations~\eqref{subrmpr} are of the same form as~\eqref{subpr}, and   curves in~\eqref{subpr} have constant geodesic curvature by Proposition ~\ref{prop:4.4}. These equations reduce to  the geodesics when $E=0$ and $F=0$, and that case corresponds to the quasi-geodesic curves.
\end{proof}


\begin{thebibliography}{99} 

\bibitem{AutMar}
C.~Autenried, I.~Markina, 
{\it Sub-Riemannian geometry of Stiefel manifolds.} 
SIAM J. Control Optim. {\bf 52} (2014), no. 2, 939-959.

%
\bibitem{AS}
A.~Agrachev, Y.~Sachkov,
{\it Control theory from the geometric point of view},
Encyclopedia of Mathematical Sciences, 
{\bf 87}
Springer-Verlag,
New York,
2004.


\bibitem{BosRos}
U.~Boscain, F.~Rossi,
{\it Invariant Carnot-Caratheodory metrics on $S^3$, $SO(3)$, $SL(2)$, and lens spaces.} 
SIAM J. Control Optim. {\bf 47} (2008), no. 4, 1851--1878. 

\bibitem{BalTysWar}
Z.~M.~Balogh, J.~T.~Tyson, B.~Warhurst, 
{\it Sub-Riemannian vs. Euclidean dimension comparison and fractal geometry on Carnot groups.} 
Adv. Math. {\bf 220} (2009), no. 2, 560--619.

%
%
\bibitem{Chow}
W.~L.~Chow,
{\it Uber Systeme von linearen partiellen Differentialgleichungen erster Ordnung.} 
Math. Ann. {\bf 117} (1939) 98--105.

%

\bibitem{Eb}
P. Eberlein, 
{\it Geometry of Nonpoitively Curved Manifolds Lie groups.} 
Cicago Lectures in Mathematics, The University of Chicago Press, Chicago, 1996.
%

\bibitem{EAS}
A.~Edelman, T.~A.~Arias, S.~T.~Smith,
{\it The geometry of algorithms with orthogonality constraints},
 SIAM J. Matrix Anal. Appl., {\bf 20}, no. 2 (1998) 303-353.
%
%

\bibitem{FJ}
Y.~Fedorov, B. Jovanovic,
{\it Geodesic flows and Newmann systems on Stiefel varieties. Geometry and integrability},
Math. Z.  {\bf 270}, no. 3-4, (2012) 659-698.

\bibitem{Hl}
S.~Helgason,
{\it Differential geometry, Lie groups and symmetric spaces,}
Academic Press, New York, 1978.


\bibitem{Jc2}
V.~Jurdjevic,
{\it Optimal control and geometry: integrable systems},
Cambridge University Press,
Cambridge Studies in Advanced  Mathematics, 
2016, Cambridge, UK.

\bibitem{JKL}
V.~Jurdjevic, K.~A.~Krakowski, F.~Silva~Leite,
{\it The geometry of quasi-geodesics on Stiefel manifolds}, (6 pages).
To appear in Proc. International Conference on Automatic Control and Soft Computing, June 4-6, 2018, Azores - Portugal.
%
%
\bibitem{KL}
K.~A.~Krakowski, L.~Machado, F.~Silva~Leite, J.~Batista,
{\it A modified Casteljau algorithm to solve interpolation
problems on Stiefel manifolds}, 
Journal of Computational and Applied Mathematics, Volume 311, (2017)
  84-99.
  
\bibitem{KMS93}
I.~Kol\'a\u{r}, P.~Michor, L.~Slov\'ak, 
{\it Natural operations in differential geometry.} 
Springer-Verlag, Berlin, 1993. pp. 434.
  
\bibitem{Mt}
R.~Montgomery,
{\it A tour of Subriemannian geometries, their geodesics and applications},
Amer Math. Soc., 2002,
Providence, Rhode Island.

%
%

\bibitem{Rashevsky}
P.~K.~Rashevski{\u\i}, 
{\it About connecting two points of complete nonholonomic space by admissible curve}, 
Uch. Zapiski Ped. Inst. K.~Liebknecht {\bf 2} (1938), 83--94.
%
%
\bibitem{St}
S.~Sternberg,
{\it Lectures on differential geometry},
Prentice- Hall, Inc, Englewood Cliffs, N.J.
1964.

\bibitem{Warner}
F.~Warner, 
{\it Foundations of differentiable manifolds and Lie groups.}
Graduate Texts in Mathematics, 94. Springer-Verlag, New York-Berlin, 1983. 272 pp.


\end{thebibliography}
\end{document}